\documentclass{article}

\usepackage[english]{babel}
\usepackage[tbtags]{amsmath}
\usepackage{amsfonts,amssymb,mathrsfs,amsmath,amscd,comment}

\usepackage{graphicx}
\graphicspath{{pictures/}} \DeclareGraphicsExtensions{.pdf,.png,.jpg}

\usepackage{amssymb}
\usepackage{amsthm}
\usepackage{amsmath}

\usepackage{cite} 
\usepackage{amsbib}

\numberwithin{equation}{section}

\theoremstyle{plain}
\newtheorem{theorem}{Theorem}[section]
\newtheorem{lemma}[theorem]{Lemma}
\newtheorem{proposition}[theorem]{Proposition}
\newtheorem{corollary}[theorem]{Corollary}
\newtheorem{assumption}[theorem]{Assumption}

\theoremstyle{definition}
\newtheorem{definition}[theorem]{Definition}
\newtheorem*{definition*}{Definition}
\newtheorem{notation}[theorem]{Notation}
\newtheorem{remark}[theorem]{Remark}
\newtheorem*{remark*}{Remark}
\newtheorem{example}[theorem]{Example}

\DeclareMathOperator{\sgn}{sgn}
\DeclareMathOperator{\sgrad}{sgrad}
\DeclareMathOperator{\grad}{grad}
\newcommand{\const}{\mathrm{const}}
\newcommand{\specialcell}[2][c]{%
\begin{tabular}[#1]{@{}c@{}}#2\end{tabular}}
\usepackage{amsmath,amssymb,amsfonts,amsthm,mathrsfs}

\begin{document}
\title{Bifurcations of Magnetic Geodesic Flows on Surfaces of Revolution}

\date{
}

\author{I.\,F.~Kobtsev$^{*,1}$ and E.\,A.~Kudryavtseva$^{**,***,2}$
\footnote{Affiliations:\\
$^{*}$Bauman Moscow State Technological University, 5/1 2-ya Baumanskaya st., Moscow, 105005 Russia, \\ 
$^{**}$ Faculty of Mechanics and Mathematics, Moscow State University, Leninskie Gory 1, Moscow, 119991 Russia, \\
$^{***}$ Moscow Center for Fundamental and Applied Mathematics, Leninskie Gory 1, Moscow, 119991 Russia \\
$^1$E-mail: {\tt int396.kobtsev@mail.ru}\\
$^2$E-mail: {\tt eakudr@mech.math.msu.su}
}
}
\maketitle
\begin{abstract}
We study magnetic geodesic flows invariant under rotations on the 2-sphere. The dynamical system is given by a generic pair of functions $(f,\Lambda)$ in one variable. Topology of the Liouville fibration of the given integrable system near its singular orbits and singular fibers is described. Types of these singularities are computed.
Topology of the Liouville fibration on regular 3-dimensional isoenergy manifolds is described by computing the Fomenko--Zieschang invariant.
All possible bifurcation diagrams of the momentum maps of such integrable systems are described.
It is shown that the bifurcation diagram consists of two curves in the $(h,k)$-plane. One of these curves is a line segment $h=0$, and the other lies in the half-plane  $h\ge0$ and can be obtained from the curve $(a:-1:k) = (f:\Lambda:1)^*$ projectively dual to the curve $(f:\Lambda:1)$ by the transformation $(a:-1:k)\mapsto(a^2/2,k)=(h,k)$.
%
\end{abstract}


\tableofcontents

\section{Introduction} 

Recently, the topology of integrable Hamiltonian systems, in particular, with two degrees of freedom, has been actively studied. It is based on the theory of classifying invariants of such systems, constructed in the works of A.\,T.~Fomenko and his school. See e.g.~\cite{Fom:86_1, Fom:86_2, Fom:87, Fom:90, BRF:2000, IGS}.
The rough classification of systems is based on the Fomenko invariant, and the Liouville classification is based on the Fomenko--Zieschang invariant, the so-called marked molecule (a graph with numerical marks). A.\,T.~Fomenko put forward a program to create an Atlas of Integrable Systems, where it is necessary to list and classify systems for which it is possible to calculate their classifying invariants. Such an Atlas is extremely useful, in particular, for detecting Liouville equivalent systems. Many such systems have already been discovered. Among them were integrable systems, which seemed to be significantly different, however, as it turned out, they are Liouville equivalent on non-singular isoenergy 3-dimensional manifolds. This circumstance reveals previously hidden connections between various problems of mechanics, physics, geometry and topology. 

This work is devoted to the topological analysis of the geodesic flow for a metric of revolution on the 2-sphere in the presence of gyroscopic forces. 
It is known that a surface of revolution is determined by a function $f=f(r)$ in one variable  $r \in [0,L]$ \cite{Kant, Tim, KZAnt}, and a magnetic field invariant under rotations on this surface is determined by another function $\Lambda=\Lambda(r)$ in the same variable \cite{KO:20, VP:2023, KK:25}. Thus the pair of functions $(f, \Lambda)$ determines a magnetic geodesic flow on this surface which turns out to be an integrable Hamiltionian system with two degrees of freedom.
For this integrable system, the semi-local singularities of ranks 0 and 1 are classified (Theorems~\ref {th3.3} and~\ref{th3.7}), provided that the pair of functions $f,\Lambda$ defining the magnetic geodesic flow is generic (see Assumption~\ref{assump2.4}), furthermore the topological Fomenko--Zieschang invariants are calculated (Theorems~\ref {th4.2} and~\ref {th4.8}). It is shown that the Liouville foliation of a magnetic geodesic flow on a sphere, in addition to 2-parameter families of regular fibers (Liouville tori), has two non-degenerate points of rank $0$ (center-center type), a finite number of 1-parameter families of critical rank-1 fibers (of elliptic and hyperbolic types) and a finite number of isolated rank-1 fibers: degenerate (Fig.~\ref {fig_2}) and non-degenerate (Fig.~\ref {fig_3}, \ref {fig_5}).

We also describe all bifurcation diagrams for magnetic geodesic flows on a sphere (Theorem~\ref{th5.17}) and establish a connection between the geometric properties of the bifurcation diagram (and the curve $(f,\Lambda)$) and dynamical properties of the system
(Theorem~\ref{th5.13} and Fig.~\ref {fig_3}, \ref {fig_7}, \ref {fig_8}, \ref {fig_9}), with which in many cases it is possible to construct the bifurcation complex from the curve $(f,\Lambda)$ or from the bifurcation diagram (Fig.~\ref{fig_10}).

\subsection{The history of the problem under investigation}
Let $M$ be a closed two-dimensional Riemannian manifold invariant 
under a $S^1$-action, and let $(q,p)$ be canonical coordinates on $T^*M$. Consider a Hamiltonian system on $T^*M$ with the symplectic form $\omega=dp\wedge dq$ and the Hamiltonian $H(q,p)=\frac{1}{2}g^{ij}(q)p_i p_j$. The question arised: what are the invariants of the Liouville equivalence of this system? In this setting, it is sufficient to study the Liouville foliation at only one energy level (i.e., calculate the Fomenko--Zieschang invariant of the Liouville foliation on $Q^3_h=\{H=h\}$), since the Liouville foliation of a geodesic flow is the same at all energy levels.

A solution to this problem for a geodesic flow on a two-dimensional sphere with a rotation metric was obtained by T.\,Z.~Nguyen and L.\,S.~Polyakova (see~\cite[vol.~2, Theor.~3.9]{IGS}) under the condition that the corresponding first integral is Bott, which is equivalent to the condition that the corresponding function $f(\theta)$ is Morse. Let's outline their main results:
\begin{itemize}
	\item the Fomenko molecule (the Reeb graph of the additional first integral restricted to an isoenergy manifold, with indicated types of bifurcations of Liouville tori at its vertices) has the form $W-W$, where $W$ is either an atom $A$ or a tree (terminal vertices are atoms $A$, non-terminal ones are saddle atoms of the form $V_l$);
	
	\item marks on any edge $A-V_l$ are $r=0$, $\varepsilon=1$;
	
	\item marks on a non-central edge $V_k-V_l$ (connecting two atoms inside $W$) are $r=\infty$, $\varepsilon=1$;
	
	\item marks on the central edge 
    $V_k-V_l$ (connecting two copies of $W$) in the case when $W$ has vertices with saddle atoms are $r=\infty$, $\varepsilon=-1$. If $W$ is just one atom $A$, then $r=1/2$, $\varepsilon=1$;
	
	\item the molecule has a unique family of saddle atoms marked with $n=2$;
	
	\item the isoenergy manifold $Q^3_h$ is diffeomorphic to $\mathbb{R}P^3$.
\end{itemize}

From the geodesic flow problem on a two-dimensional sphere $M$ with a rotation metric, a number of similar statements can be obtained by changing the Hamiltonian or the manifold $M$ itself.

The first method can be implemented, for example, by adding potential to the geodesic flow of the sphere, as was done by E.\,O.~Kantonistova in~\cite{Kant}. Here are the main results of this work.

Let $M^2\cong S^2$ be a two-dimensional Riemannian manifold of revolution.
One can introduce semi-geodesic coordinates $(r,\varphi)$ on it, in which the metric has the form $ds^2=dr^2+f^2(r)d\varphi^2$. Let $(q,p)$ be canonical coordinates on $T^*M$ with symplectic form $\omega=dp\wedge dq$ and the Hamiltonian $H(q,p)=\frac{1}{2}g^{ij}(q)p_i p_j+V(q)$. In the coordinates $q=(r,\varphi)$, they have the form $\omega=dp_r\wedge dr + dp_\varphi\wedge d\varphi$, $H=\dfrac{p_r^2}{2}+\dfrac{p_\varphi^2}{2f^2(r)}+V(r)$.

Under certain conditions on the functions $f(r)$, $V(r)$, the following results were obtained in~\cite{Kant} compared to the case of geodesic flow on the sphere:
\begin{itemize}
\item the mark $n$ on the single family can be equal to $0$, $1$ or $2$;
	
\item $Q^3_h$ is diffeomorphic to $S^3$, $S^1\times S^2$ or $\mathbb{R}P^3$ respectively;
    
\item singular points of rank $0$ are described, a classification of degenerate and non-degenerate singularities of rank $1$ is given, and properties of bifurcation diagrams are described.
\end{itemize}

The second way is to take another configuration manifold (while keeping the requirement of $S^1$-invariance of the metric and the potential). This approach was applied by D.\,S.~Timonina~\cite{Tim}.  For systems on the torus and the Klein bottle, the molecules were constructed and the marks were computed. In this problem, atoms with stars ($A^*$, $A^{**}$) appeared for the first time. Thus, in~\cite{Tim}, the Fomenko--Zieschang invariants for these integrable systems were completely computed.

The problem of the motion of a particle in a potential field on a projective plane was considered by E.\,I.~Antonov and I.\,K.~Kozlov in~\cite{KZAnt}. In this work, the ideas of V.\,S.~Matveev (\cite[vol.~2, Theor.~3.11]{IGS}, the case of zero potential on the projective plane) and D.\,S.~Timonina are developed  \cite{Tim}. In this formulation of the problem, new properties also appeared -- in particular, the appearance of the mark $r=1/4$ (previously there were only $0$, $1/2$, $\infty$) and the appearance of an isoenergy manifold diffeomorphic to the lens space $L_{4,1}$. Thus, in~\cite{KZAnt}, the Fomenko--Zieschang invariants on the nonsingular isoenergy manifold $Q^3_h$ are completely calculated and the topology of such $Q^3_h$ is found.

The article~\cite{KZO:E3} was devoted to a related problem: the study of an integrable system with a linear in momenta periodic first integral on the Lie algebra $e(3)$. In this work, non-degenerate singularities of ranks $0$ and $1$ were classified, bifurcations of Liouville tori on $Q^3_h$ were described, the topology of $Q^3_h$ was found.

The paper~\cite{KO:20} studied the problem obtained from the problem posed in~\cite{Kant} by adding a magnetic field, i.e., gyroscopic forces (while preserving the requirement of $S^1$-invariance of the Riemannian metric, magnetic field and potential). Under certain ``genericity'' conditions, non-degenerate and degenerate singularities of ranks $0$ and $1$ were classified~\cite{KO:20} in terms of the triple of functions defining this system. The Fomenko molecules for $Q^3_h$ were found~\cite{KO:20} (including a description of bifurcations of Liouville tori). Marks on the edges of the molecule connecting vertices with saddle 3-atoms, and the topology of $Q^3_h$ were also described~\cite{KO:20}.

In this paper, a more detailed study of the problem from~\cite{KO:20} is given in the special case when the potential is zero. In particular, all marks for the Fomenko--Zieschang invariants encoding the Liouville foliation on isoenergy surfaces have been calculated (thereby completing the calculation of marks started in~\cite{KO:20}). In addition, a number of issues have been resolved that were not investigated in the work \cite{KO:20} (see Theorems~\ref{th5.13} and~\ref {th5.17};
the results obtained in them are described in the beginning of Introduction).

In solving this problem, several new interesting phenomena have been discovered in this work, showing that such a problem is of interest. First, it is shown that various magnetic geodesic flows are characterized (given) by a planar curve, which is essentially arbitrary except for some boundary conditions at its end points (this is illustrated with two examples in Fig.~\ref {fig_10}). Thus, all the invariants of interest can be described in terms of this curve (Theorems~\ref{th3.3}, \ref {th3.7}, \ref{th4.2}, \ref {th4.8}, \ref {th5.13} and~\ref {th5.17}). At the same time, a somewhat unexpected geometric fact was discovered: to describe these invariants, it is useful to switch to a projectively dual curve (section \ref{su5.1}). Secondly, a new type of degenerate singularities has been discovered, a so-called ``asymmetric elliptic fork'' (Theorem~\ref{th3.7}), which was not mentioned in other works, including~\cite{KO:20}. Thus, even without a potential, the system turns out to be very rich.

\section{Magnetic geodesic flow $S(f,\Lambda)$ on a spherical surface of revolution} \label{s2}

Let's proceed with a more detailed description of the dynamical system studied in this paper. This system can be obtained on the basis of a geodesic flow on a two-dimensional manifold by introducing additional gyroscopic forces. 
Since the action of a magnetic field on a charged particle is described in this way, we will call such a system a ``magnetic geodesic flow'' on a two-dimensional surface of revolution. 

Let's give a strict definition of this concept.
\begin{definition}
    {\em Magnetic geodesic flow} on a Riemannian manifold $(M,g)$ is a dynamical system on $T^*M$ with Hamiltonian $H=\dfrac{1}{2}g^{ij}(q)p_i p_j$ and symplectic structure $\widetilde{\omega}=dp \wedge dq+\beta$, where $q=(q^1,q^2)$ are local coordinates on $M$, $p=(p_1,p_2)$ are the corresponding conjugate momenta, $g^{ij}$ is the matrix inverse to the matrix of the metric $g$, and $\beta$ is a closed 2-form on $M$ (determining a magnetic field).
\end{definition}

The choice of the manifold $M$ is restricted by the conditions of V.\,V.~Kozlov's theorem in~\cite{VKZ:79}: a first integral of a geodesic flow, independent with $H$, exists only on two-dimensional surfaces of genus $0$ and $1$. This means that $M$ can be diffeomorphic to either a sphere, or a torus, or a projective plane, or a Klein bottle. In this paper, the first case is studied, and the second case will be studied in \cite{KK:25}.

Now let $(M,ds^2)$ be a two-dimensional Riemannian manifold diffeomorphic to the sphere $S^2$. Let's also assume that the Riemannian metric on $M$ is invariant under rotations, i.e., a smooth $S^1$-action is defined on $M$ by isometries. As one knows, such an effective $S^1$-action has exactly two fixed points; let's call them the north and south poles ($N$ and $S$ respectively). Let $L$ be the length of a geodesic connecting the poles. One can show~\cite[Prop.~4.6 (ii)]{Besse:78} that the Riemannian metric $ds^2$ on $M\setminus\{N,\,S\}$ can be written as $ds^2=dr^2+f^2(r)d\varphi^2$, where $r\in (0,L)$, $\varphi\mod2\pi$ is an angular coordinate,
$f: \left[0,L\right] \to\mathbb{R}_{+}$ is a smooth function satisfying the conditions
\begin{equation}
\label{eq2.1}
    f|_{(0,L)}>0, \qquad f^{(2j)}(0)=f^{(2j)}(L)=0, \quad j\in\mathbb Z_+, \qquad f'(0)=1, \quad f'(L)=-1.
\end{equation}
The constructed coordinate system $(r,\varphi) \in (0,L)\times S^1$ on $M\setminus\{N,\,S\}$ is called {\em semi-geodesic}.

Let's also assume that $M$ has a smooth 2-form $\beta$ that is invariant under the $S^1$-action. It is known (see, for example, \cite[Lemma~1]{KO:20}) that, in semi-geodesic coordinates, it has the form $\beta=\Lambda'(r)dr\wedge d\varphi$, where $\Lambda: [0,L]\to\mathbb R$ is a smooth function satisfying the conditions
\begin{equation} \label{eq2.2}
\Lambda^{(2j+1)}(0) = \Lambda^{(2j+1)}(L) = 0, \quad j\in\mathbb Z_+.
\end{equation}

\begin{definition} \label{def2.2}
The magnetic geodesic flow on $M$ is a dynamical system on $T^*M$ with the Hamiltonian $H$ and the symplectic structure $\widetilde{\omega}$ on $T^*M$, which have the following form in coordinates corresponding to semi-geodesic coordinates on $M\setminus\{N,\,S\}$:
\begin{equation}
\label{eq2.3}
    H=\dfrac{p_r^2}{2}+\dfrac{p_\varphi^2}{2f^2(r)}, \quad \widetilde{\omega}=dp_r \wedge dr + dp_\varphi \wedge d\varphi+\Lambda'(r)dr \wedge d\varphi.
\end{equation}
The latter summand in the expression~\eqref{eq2.3} of the 2-form $\widetilde{\omega}$ is the 2-form $\beta=\Lambda'(r)dr\wedge d\varphi$ of the gyroscopic forces, the so-called 2-form of the {\em magnetic field}.
\end{definition}

From now on, we will denote such a system by $S(f,\Lambda)$ for brevity. It is directly verified that this system has an additional first integral 
$$
K=p_\varphi+\Lambda(r),
$$
and the Hamiltonian flow generated by it is $2\pi$-periodic. This proves that the magnetic geodesic flow is a completely Liouville-integrable Hamiltonian system.

\begin{example} \label{ex2.3}
For example, consider two classes of magnetic geodesic flows on a sphere.
\begin{itemize}
    \item[(a)] If $\Lambda'(r)=\lambda f(r)$ where $\lambda=\const$, then the magnetic field is proportional to the area form and we get a {\em uniform magnetic geodesic flow}.
    \item[(b)] If $\Lambda(0)=\Lambda(L)$, then the magnetic field is {\em exact} (i.e., $\beta=dA$, where $A=(\Lambda(r)-\Lambda(0)) d\varphi$ is a smooth 1-form, the so-called magnetic potential), and the mapping $(p_r,\,p_\varphi,\,r,\,\varphi) \mapsto(p_r,\,\widetilde p_\varphi=K-\Lambda(0),\,r,\,\varphi)$ extends to a diffeomorphism of $T^*M$ onto itself, transforming the symplectic structure $\widetilde\omega$ into the standard form $\omega=dp_r\wedge dr +d\widetilde p_\varphi\wedge d\varphi$. Let's replace $p_\varphi\to\widetilde p_\varphi=\widetilde K$ and,
    by abusing notations, we will denote $\widetilde{K}$ by $K$. 
    After this change, the function $H$ and the symplectic structure will take the form
\begin{equation} \label{eq2.4}
H=\dfrac{p_r^2}{2}+\dfrac{(K + \Lambda(0)-\Lambda(r))^2}{2f^2(r)}, \quad \omega=dp_r \wedge dr + dK \wedge d\varphi.
\end{equation}
In the case of $\Lambda(0)=\Lambda(L)$, the representations~\eqref{eq2.3} and~\eqref{eq2.4} of the magnetic geodesic flow are equivalent (where each representation is considered together with its extention from $T^*(M\setminus\{N,S\})$ to $T^*M$ by continuity).
\end{itemize}
For any magnetic geodesic flow $S(f,\Lambda)$, the representation~\eqref{eq2.4} can be extended to $T^*(M\setminus\{S\})$. For $T^*(M\setminus\{N\})$, one should use the representation of the magnetic potential in the form of $(\Lambda(r)-\Lambda(L)) d\varphi$.
The intersection of two classes from Example \ref {ex2.3}, i.e., uniform and exact magnetic geodesic flows, consists of geodesic flows on surfaces of revolution, which means that the magnetic field vanishes.
Indeed, for uniform magnetic field, we have $\Lambda'(r)=\lambda f(r)$ where $\lambda=\const$, therefore $\Lambda(L)-\Lambda(0)=\lambda\int\limits_{0}^{L}f(s)ds$. Due to~\eqref{eq2.1}, this integral is positive, thus the exactness condition is equivalent to $\lambda=0$, thus $\beta=0$.
\end{example}

We will not impose any other restrictions on the system, except for the smoothness conditions~\eqref{eq2.1} and~\eqref{eq2.2} and some genericity conditions (see conditions 1--6 in
Assumption~\ref{assump2.4} below). In particular, we will not require that the magnetic field is uniform or exact (see Example \ref {ex2.3}). The case of uniform exact magnetic geodesic flows of billiard type on some piecewise flat surfaces of revolution (homeomorphic to a sphere, torus, disk and cylinder) was studied in~\cite{VP:2023}.

Note that we do not require that the Riemannian manifold $(M,g)$ has an isometric immersion into $\mathbb R^3$ as a surface of revolution. Such an isometric immersion exists if and only if $|f'(r)|\le1$ for all $r\in[0,L]$ \cite{Besse:78}. Note that, after regular changes of the parametrization $r=r(\widetilde r)$ of the planar curve $(f,\Lambda)$, the topology of the Liouville foliation of the magnetic geodesic flow is preserved (Theorem~\ref{th5.17}), although Riemannian manifolds obtained in this way are not isometric to the original one. If we take the natural parameter $s=s(r)$ on the curve $(f,\Lambda)$ as the parameter $\widetilde r$, then the corresponding Riemannian manifold $(M,g_s)$ will already allow an isometric immersion into $\mathbb R^3$ as a surface of revolution obtained by rotating the planar curve $(f,\Lambda)$ around the line $f=0$. This surface of revolution is useful for describing the types of non-degenerate critical circles of the initial magnetic geodesic flow (see Theorems~\ref{th3.7},~\ref{th5.13} and Remark~\ref{rem5.15}).

\begin{assumption} \label{assump2.4}
It follows from~\eqref{eq2.1} and~\eqref{eq2.2}
that the functions $f,\Lambda$ determine a smooth Riemannian metric and a smooth magnetic field on a sphere (and therefore a smooth integrable Hamiltonian system $S(f,\Lambda)$) if and only if the following constraints are fulfilled:
\begin{enumerate}
    \item[1.] $f(r)>0$, $r\in(0,L)$, extends to a smooth odd $2L$-periodic function on $\mathbb{R}$;
    \item[2.] the function $\Lambda(r)$ extends to a smooth even $2L$-periodic function on $\mathbb{R}$;
    \item[3.] $f'(0)=1$, $f'(L)=-1$.
\end{enumerate}
Additionally, we will assume that a pair of $2L$-periodic functions $f,\Lambda$ is {\em generic} in the following sense:
\begin{enumerate}
    \item[4.] $f(r)$, $\Lambda(r)$ are Morse functions;
    \item[5.] $(\Lambda'(r))^2+(f'(r))^2>0$ (i.e., the planar curve $(f(r),\Lambda(r))$, $0\le r\le L$, is regular);
    \item[6.] the function $f'\Lambda''-f''\Lambda'$ has only simple zeros (geometrically, this means that the oriented curvature of the planar curve $(f,\Lambda)$ has only simple zeros, i.e., all straightening points of the curve $(f,\Lambda)$ are inflection points).
\end{enumerate}
In Lemma~\ref{le5.11}, we will give conditions (i)--(iv) equivalent to the conditions 1--6.
\end{assumption}

\begin{remark*}
It is not difficult to show that pairs of functions satisfying the conditions 1--6 from Assumption~\ref{assump2.4} {\em(genericity 
conditions)} form an open dense subset in the space of pairs of functions satisfying the conditions~\eqref{eq2.1} and~\eqref{eq2.2}. The authors do not know whether a similar property is fulfilled for the conditions imposed on functions in the works~\cite{Kant, KO:20}.
\end{remark*}

\section{Semi-local and semi-global 4D singularities of the magnetic geodesic flow $S(f,\Lambda)$} \label {s3}

Recall that the mapping 
$$
\mathcal{F}=(H,K):T^*M\to\mathbb R^2
$$
is called the {\em momentum mapping} of the given integrable system.
Consider the {\em Liouville foliation} on $T^*M$ associated with the given integrable system, whose fibers are connected components of the integral submanifolds $\{H=\const,\ K=\const\}$. Two integrable systems are called {\em Liouville equivalent} (and the corresponding Liouville foliations have the same topology) if there exists a homeomorphism of their phase spaces transforming fibers of one foliation to fibers of the other.
The base of the Liouville foliation is called the {\em bifurcation complex}; it was introduced by A.T.~Fomenko (the cell-complex $K$ in \cite[\S 5]{Fom:88}, the affine variety $A$ in \cite{EG:2012} called the {\em unfolded momentum domain}); it is a (branched) covering over the image of the momentum map \cite{EG:2012}.

By {\em rank} of a point $(p,q)\in T^*M$, we will mean the rank of the mapping $d\mathcal{F}$ at this point. Points of ranks 0 and 1 are called {\em critical} or {\em singular} points of the mapping $\mathcal{F}$. The image of the set of critical points under the mapping $\mathcal{F}$ is called the {\em bifurcation diagram} of the momentum mapping.
For the definition of a {\em non-degenerate} critical point and its type, see~\cite{IGS}.

Consider the Hamiltonian $\mathbb R^2$-action on $T^*M$ generated by the first integrals $H,K$. Clearly, each orbit of this action is contained in a fiber. 
A fiber (respectively orbit) is called {\em singular} if it contains a singular point.

By a {\em local} (respectively {\em semi-local}, {\em semi-global}~\cite{vu2003}) {\em singularity}~\cite{zung96a, zung03a, bol:osh06, Hansmann, BGK:2018} of the Liouville foliation, we will mean the germ of the Liouville foliation at a singular point (respectively orbit, fiber). These are 4D singularities. For the Liouville foliation on a regular isoenergy 3-manifold $Q^3_h$, one obtains 3D singularities.
Since Liouville equivalent systems have the same bifurcation complex, it is a global 4D topological invariant.

\subsection{Semi-local 4D singularities}

\begin{definition}
By the {\em effective potential} of the system $S(f,\Lambda)$, we mean the function 
\begin{equation}
\label{eq3.1}
    U_k(r)=\dfrac{(k-\Lambda(r))^2}{2f^2(r)},
\end{equation}
where $k$ is a real parameter.
\end{definition}

\begin{notation} \label {not3.2}
Throughout the rest of the paper, we will denote:
\begin{itemize}
    \item $\{r_i\}_{i=1}^n$ are critical points of the function $f$ on $\left[0,L\right]$, $\{r^*_j\}_{j=1}^N$ are critical points of the function $\Lambda$ on $\left[0,L\right]$, $\{r_\ell^\circ\}_{\ell=1}^m$ are zeros of the function $f'\Lambda''-f''\Lambda'$ on $\left[0,L\right]$ (the points of each of the three sets are listed in ascending order), 
    the conditions 1--6 of Assumption~\ref{assump2.4} imply that these three sets do not intersect pairwise;

    \item $I=\left[0,L\right] \setminus \{r_i\}_{i=1}^n$;

    \item $\mathcal{O}_{\rho,k} = \{ (p_r, \,K, \, r, \, \varphi) : p_r=0,\,K=k,\,r=\rho,\,\varphi\in \mathbb{R} / 2\pi \mathbb{Z}\}$, a circle depending on a pair of parameters $\rho\in(0,L)$, $k\in \mathbb{R}$, which is an orbit of the Hamiltonian $S^1$-action generated by the first integral $K$.
\end{itemize}
\end{notation}

\begin{theorem}[{\cite[Propositions~1 and~2~(A)]{KO:20}}] \label{th3.3}
If the functions $f,\Lambda$ satisfy the genericity conditions~1--5 from Assumption~\ref{assump2.4}
then the singular points of the magnetic geodesic flow $S(f,\Lambda)$ on the sphere are exhausted by the following list:
\begin{enumerate}
    \item two singular points $(0,\, N)$ and $(0,\,S)$ of rank 0, whose images under the momentum mapping $\mathcal F$ are the points $(0,\Lambda(0))$ and $(0,\Lambda(L))$. They are non-degenerate and have a center-center type;
    \item two families of singular circles of rank 1 (which are $S^1$-orbits from Notation~\ref{not3.2}):
    \begin{itemize}
        \item $\mathcal{C}_1^k=\bigcup\limits_{r\in I\cap(0,L)} \mathcal{O}_{r,k(r)}$, where $k(r)=\Lambda(r)-f(r) \dfrac{\Lambda'(r)}{f'(r)}$,
        \item $\mathcal{C}_1^\Lambda=\bigcup\limits_{r\in (0,L)} \mathcal{O}_{r,\Lambda(r)}$.
    \end{itemize}
    The image of the family $\mathcal{C}_1^k$ under the momentum mapping $\mathcal F$ is the curve $\gamma_1 =\{\left(h(r),\,k(r)\right)\mid r\in I\cap(0,L)\}$, where $h(r)=\dfrac{1}{2}\left(\dfrac{\Lambda'(r)}{f'(r)} \right)^2$, and the image of the family $\mathcal{C}_1^\Lambda$ is the vertical interval $\gamma_2 = \{\left(0, \Lambda(r) \right)\mid r\in(0,\,L)\}$.
\end{enumerate}
\end{theorem}

\begin{remark}
(a) The first statement of Theorem~\ref{th3.3} (about the type of the singular points of rank $0$) will remain true even if the condition $(\Lambda'(r))^2+(f'(r))^2>0$ is not fulfilled. Both statements of the theorem will remain true even if, instead of the conditions~1--5 from Assumption~\ref{assump2.4}, the following weaker conditions are satisfied: \eqref{eq2.1}, \eqref{eq2.2}, $\Lambda''(0)\Lambda''(L)\ne0$ and $(f'(r))^2+(\Lambda'(r))^2>0$.

(b) Let us explain how magnetic geodesics behave near singular points from Theorem~\ref{th3.3}. Unlike geodesics, a magnetic geodesic is determined not only by the starting point and the tangent at that point, but also by the speed at this point. Near points of rank 0, we get slow rotations near the poles of the sphere: the magnetic geodesic is close to a circle of small radius, and the direction of rotation along this circle is determined by the sign of $\Lambda''(0)$ or $\Lambda''(L)$, respectively.  
The critical circle of each family from Theorem~\ref{th3.3} is a parallel $\{r=\const\}$. The critical circle of the first family $\mathcal C_1^k$ corresponds to the motion along a parallel with constant angular velocity $\frac{d\varphi}{dt}=\omega(r)$, uniquely determined by the Newton second law: the covariant acceleration (equal to $-\omega^2(r)f(r)f'(r)\frac{\partial}{\partial r}$) coincides with the magnetic force (equal to $\omega(r)\Lambda'(r)\frac{\partial}{\partial r}$), thus $\omega(r)=-\frac{\Lambda'(r)}{f(r)f'(r)}$. The second family  $\mathcal C_1^\Lambda$ consists of equilibria and coincides with the zero section of the bundle $T^*(M\setminus\{N,S\})$.
\end{remark}

\begin{proof}
Let's change the coordinates $(p_r, \, p_\varphi, \, r, \, \varphi) \to (p_r, \, K=p_\varphi+\Lambda(r), \,r,\, \varphi)$ (recall that these coordinates are defined only on $T^*(S^2\setminus\left\{N,S\right\}$)) and compute $dH$ and $dK$ in them:
\begin{equation}
\label{eq3.2}
    \begin{split}
	&dH = \left( p_r, \, \frac{K-\Lambda(r)}{f^2(r)}, \, -\dfrac{\left(K-\Lambda(r)\right)\left((K-\Lambda(r))f'(r)+\Lambda'(r)f(r)\right)}{f^3(r)}, \, 0 \right),\\
	&dK = \left(0,\,1,\,0,\,0\right). \\
    \end{split}
\end{equation}
Obviously, there are no rank 0 singularities outside the poles.
	
Let us study the behavior of the first integrals $H$ and $K$ at the poles. Without loss of generality, take the point $N$ ($r=0$) and introduce new coordinates $x=r\cos\varphi$, $y=r\sin\varphi$ on $M$ in its neighbourhood. In the tangent space at the pole, we take the canonical conjugate momenta $p_x$, $p_y$.
We have $p_\varphi=x p_y-y p_x$, $r p_r = x p_x + y p_y$. Near the point $N$, we have $f(r)=r+o(r)$, $\Lambda(r)=l_0+\frac{l_1}{2}r^2+o(r^2)$, due to the properties 1--5 from Assumption~\ref{assump2.4}. Let's write the first integrals in the new coordinates:
\begin{equation}
\label{eq3.3}
    \begin{split}
	&H=\frac{1}{2}(p_x^2+p_y^2) (1+o(x^2+y^2)), \\
	&K=x p_y-y p_x + l_0+\frac{l_1}{2}(x^2+y^2)+o(x^2+y^2). 
    \end{split}
\end{equation}
On the plane, we have $dx\wedge dy=r dr\wedge d\varphi$ (as the area form), so the following relations are valid:
\begin{equation*}
    \Lambda'(r)dr \wedge d\varphi=\dfrac{\Lambda'(r)}{r}r dr \wedge d\varphi=\dfrac{l_1 r+o(r)}{r} r dr \wedge d\varphi=(l_1+o(1))dx\wedge dy.
\end{equation*}
In the new coordinates, the symplectic structure at any point of the form $(p,N)$ (over the pole $N$) takes the form
\begin{equation*}
    \widetilde{\omega}=\left(
	\begin{matrix}
	   0 & 0 & 1 & 0 \\
	   0 & 0 & 0 & 1 \\
	   -1 & 0 & 0 & l_1 \\
	   0 & -1 & -l_1 & 0 \\
	\end{matrix}
    \right).
\end{equation*}
It follows from~\eqref{eq3.3} that $dH\big|_{(p,N)}
=(p_x,p_y,0,0), \quad dK\big|_{(p,N)}=(0,0,p_y,-p_x)$.
	
A point $(p,N)$ has rank zero if and only if $p=0$. In order to determine the type of this singularity, we compute the Hamiltonian operators $A_H= \widetilde\omega^{-1}d^2 H$, $A_K=\widetilde\omega^{-1} d^2 K$ at the point $(0,N)$ in coordinates $\left(p_x,p_y,x,y\right)$:
\begin{equation*}
    A_H=\left(
        \begin{matrix}
	   0 & l_1 & 0 & 0 \\
	   -l_1 & 0 & 0 & 0 \\
	     1 & 0 & 0 & 0 \\
	   0 & 1 & 0 & 0 \\
	\end{matrix}
    \right), \quad	A_K=\left(
	\begin{matrix}
	   0 & -1 & 0 & 0 \\
	   1 & 0 & 0 & 0 \\
	   0 & 0 & 0 & -1 \\
	   0 & 0 & 1 & 0 \\
	\end{matrix}
    \right).
\end{equation*}
After making sure that the commutative subalgebra generated by them is two-dimensional, we compose the operator $\lambda A_H + \mu A_K$:
\begin{equation*}
\lambda A_H + \mu A_K=\left(
    \begin{matrix}
	0 & -\mu+\lambda l_1 & 0 & 0 \\
	\mu-\lambda l_1 & 0 & 0 & 0 \\
	\lambda & 0 & 0 & -\mu \\
	0 & \lambda & \mu & 0 \\
    \end{matrix}
    \right).
\end{equation*}
Its eigenvalues are $\pm i\mu$, $\pm i (\mu-\lambda l_1)$. Therefore, since $l_1\neq 0$ (which is true, since $\Lambda(r)$ is a Morse function), the singularity is non-degenerate and has the center-center type.

The second pole of the sphere (the point $S$) is investigated by replacing $r\to L-r$.

Let's move on to the singularities of rank 1. Let $f'(r^*)=0$, then $\Lambda'(r^*)\neq 0$, $f(r^*)\neq 0$, thus~\eqref{eq3.2} gives $K=\Lambda(r^*)$, $p_r=0$ (and $H=0$).

If $f'(r)\neq 0$ then in view of~\eqref{eq3.2}, we have
\begin{equation}
\label{eq3.4}
    \dfrac{\partial H}{\partial r}=-\dfrac{f'(r)}{f^3(r)}(K-k(r))\left(K-\Lambda(r)\right).
\end{equation}
The equality $K-k(r)=0$ defines a 1-parameter family (with parameter $r$) of critical points $\mathcal{C}_1^k$. Similarly, the equality $K-\Lambda(r)=0$ yields the family $\mathcal{C}_1^\Lambda$. By substituting the parametrizations of the families $\mathcal{C}_1^k$, $\mathcal{C}_1^\Lambda$ into the expressions of the first integrals $H$, $K$ in the coordinates $(p_r,K,r,\varphi)$, we obtain the curves $\gamma_1$, $\gamma_2$, respectively.

The theorem is proved.
\end{proof}

Let us study the topology of the Liouville foliation near the semi-local singularities of rank 1 found in Theorem~\ref{th3.3}. The following statement easily follows from the definition of non-degenerate singular points of rank 1 for systems with 2 degrees of freedom.

\begin{proposition} \label{propos3.5}
A singular point $x$ of rank 1 of the momentum mapping (and the phase trajectory containing it) is non-degenerate if and only if the matrix
\begin{equation*}
    J=\left(
	\begin{matrix}
	\frac{\partial^2 H}{\partial p_r^2} & \frac{\partial^2 H}{\partial p_r \partial r} \\
	\frac{\partial^2 H}{\partial p_r \partial r} & \frac{\partial^2 H}{\partial r^2} \\ 
	\end{matrix}
    \right)
\end{equation*}
is non-degenerate at the point $x$. If $\det J(x)>0$ then the trajectory containing this point is called elliptic (and it is orbitally stable in the case of $dH(x)\neq 0$. If $\det J(x)<0$ then it is called hyperbolic (and it is orbitally unstable in the case of $dH(x)\neq 0$).
\end{proposition}

\begin{remark} \label{rem3.6}
Along with non-degenerate rank 1 singularities of the Liouville foliation mentioned in Proposition \ref{propos3.5}, we will also
meet degenerate rank 1 singularities of the following two types. The first type is a parabolic singularity, in a neighbourhood of which the Liouville foliation is given by the first integrals of the form $2\widetilde H=p_1^2 + q_1^3 - p_2q_1$ and $\widetilde K=p_2$ (in some coordinates $(p_1, p_2, q_1, q_2) \in D^3\times (\mathbb{R}/ 2\pi\mathbb{Z})$, in which the parabolic orbit has the form $p_1=p_2=q_1=0$). The second type is a singularity of the ``elliptical fork'' type, in a neighbourhood of which the Liouville foliation is given by the first integrals of the form $2\widetilde H=p_1^2 + q_1^4 - 2p_2q_1^2$ and $\widetilde K=p_2$ (in coordinates of the same type as for the parabolic singularity). 
\end{remark}
    
\begin{theorem}	\label{th3.7}
Let the functions $f,\Lambda$ satisfy the conditions 1--5 from Assumption~\ref{assump2.4}. Then the types of critical circles of the magnetic geodesic flow $S(f,\Lambda)$ on the sphere are determined by the following conditions: 
\begin{enumerate}
    \item For any $r\in (0,\,L)$ such that $f'(r)\neq 0$, the following conditions {\rm(a)--(d)} are equivalent:

        {\rm(a)} the critical orbit $\mathcal{O}_{r,k(r)}$ from the family $\mathcal{C}_1^k$ is non-degenerate,
		
		{\rm(b)} $U''_{k(r)}(r) \neq 0$, 
		
		{\rm(c)} $k'(r)\Lambda'(r) \neq 0$, 

        {\rm(d)} $(f'(r)\Lambda''(r)-f''(r)\Lambda'(r))\Lambda'(r)\ne0$,

    where $U_k(r)$ is the effective potential~\eqref{eq3.1}, $k(r)$ is the function defining the family $\mathcal{C}_1^k$, $\gamma_1(r)=(h(r),k(r))$ is the bifurcation curve from Theorem~\ref{th3.3}.
    If $U''_{k(r)}(r) > 0$ then the critical circle has an elliptic type, if $U''_{k(r)}(r) <0$ then it is hyperbolic; in this case, $\sgn U''_{k(r)}(r) =\sgn(-k'(r)\Lambda'(r)) =\sgn(f'(r)\Lambda''(r)-f''(r)\Lambda'(r))\Lambda'(r))$.\footnote{The sign of the value $(f'(r)\Lambda''(r)-f''(r)\Lambda'(r))\Lambda'(r)$ coincides with the sign of the Gaussian curvature of the surface obtained by rotating the planar curve $(f,\Lambda)$ around the straight line $\{f=0\}$, evaluated at this point (Fig.~\ref{fig_7}--\ref{fig_9}). The Gaussian curvature of this surface is $-\frac1f\,\frac{d^2f}{ds^2}=\frac{(f'\Lambda''-f''\Lambda')\Lambda'}{({f'}^2+{\Lambda'}^2)^2f}$ where $s=s(r)$ is the natural parameter on the curve $(f,\Lambda)$.\label{footnote1}}
		
    \item Any critical orbit $\mathcal{O}_{r,\Lambda(r)}$ of the system $S(f,\,\Lambda)$ contained in $\mathcal{C}_1^\Lambda\setminus (\mathcal{C}_1^\Lambda\cap\mathcal{C}_1^k)$, is non-degenerate and has an elliptic type.
		
    \item If the functions $f,\Lambda$ also satisfy the condition~6 then degenerate critical circles $\mathcal{O}_{r,k(r)}$ of the family $\mathcal{C}_1^k$ can be of one of the following two types:
    \begin{itemize}
        \item a parabolic circle~\cite{BRF:2000, EG:2012, BGK:2018, Ler:87, Ler:94, KM:21},
        provided that $r\in I$ is such that $f'(r)\Lambda''(r)-f''(r)\Lambda'(r)=0$ (an inflection point of the curve $(f,\Lambda)$); its image under the momentum mapping is an ordinary cusp point
        (having a semi-cubic parabola type)
        of the bifurcation curve $\gamma_1$. The germ of the Liouville foliation at the corresponding degenerate orbit $\mathcal O_{r_\ell^\circ,k(r_\ell^\circ)}$ has the type ``parabolic orbit'';

        \item a critical orbit contained in both families $\mathcal{C}_1^k$ and $\mathcal{C}_1^\Lambda$, provided that $r\in I\cap(0,L)$ is such that $\Lambda'(r)=0$ (a point of the curve $(f,\Lambda)$ with a horizontal tangent); its image under the momentum mapping is the tangency point of the bifurcation curves $\gamma_1$ and $\gamma_2$. The germ of the Liouville foliation at the corresponding degenerate orbit $\mathcal O_{r_j^*,\Lambda(r_j^*)}$ is of the type ``asymmetric elliptic fork''~\cite{BRF:2000, Ku:22, KK:25}.
    \end{itemize}
\end{enumerate}
\end{theorem}

\begin{remark*}
(a) According to item~1 of Theorem~\ref{th3.7}, the type of critical circles from the family $\mathcal{C}_1^k$ is determined by the sign of the value from condition (d). The points $r_j^*$ and $r_\ell^\circ$ divide the domain of the parameter $r$ into subintervals, at which either $\Lambda'(r_i)=0$ or $f'(r)\Lambda''(r)-f''(r)\Lambda'(r)=0$. Since the sets $\{r_j^*\}$ and $\{r_\ell^\circ\}$ do not intersect, exactly one of the factors in the condition (d) changes the sign when 
the parameter $r$ on the bifurcation curve $\gamma_1$ passes through a value $r_j^*$ or $r_\ell^\circ$, so the type of the corresponding critical circles (elliptic or hyperbolic) changes.

(b) If the magnetic field coincides with the area form (see Example~\ref{ex2.3}~(a)), then elliptic forks do not occur, since $\Lambda'(r)=f(r)>0$ outside the poles, thus the function $\Lambda(r)$ has no critical points.

(c) In the work \cite{KO:20}, elliptic forks did not arise if the triple of functions $f,\Lambda,U$ satisfies some genericity conditions, where $U=U(r)$ is the potential. As one can see from Theorem~\ref {th3.7}, elliptic forks arise in the situation under consideration (in the absence of potential, i.e., for $U\equiv0$). This is because the triples of functions $f,\Lambda,U$ with $U\equiv0$ do not satisfy the genericity conditions from the work~\cite{KO:20}. It can be shown that, by adding a potential $U$, it is possible to destroy these elliptic forks, see Fig.~\ref{fig_1}.
\end{remark*}

\begin{figure}[ht!] 
\begin{minipage}[r]{0.27\linewidth}
    \begin{center} 
    \includegraphics[scale=1.9]{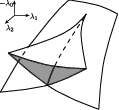} \\
   \textit{a)}
    \end{center}
\end{minipage} 
\begin{minipage}[r]{0.72\linewidth}
    \begin{center}
        \includegraphics[scale=.33]{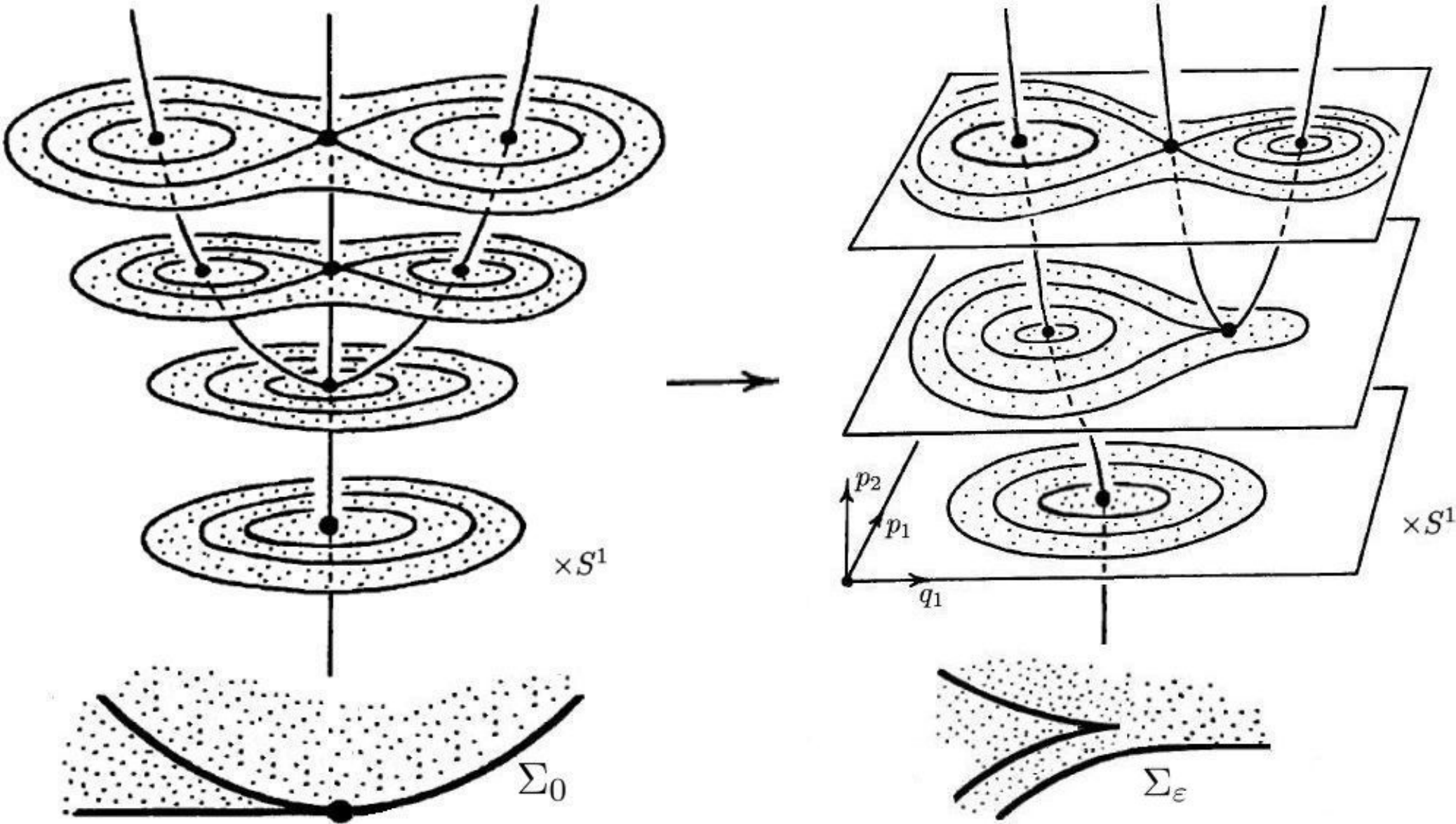} \\ \textit{b)}
    \end{center}
\end{minipage} 
\caption{\textit{a)} Bifurcation diagram $\Sigma^\circ$ in the space $(\lambda_0,\lambda_1,\lambda_2)=(-h,\varepsilon,p_2(K))$
for a singularity of type $A_3$ (swallowtail); 
\textit{b)} appearance of a cuspidal torus in the magnetic geodesic flow
with the potential $\varepsilon U$ near an ``asymmetric elliptic fork'' of the system $S(f,\Lambda)$ with a small $\varepsilon=\lambda_1>0$}
 \label {fig_1}
 \end{figure}

\begin{proof}
Let's prove the first statement of the theorem.
	
Let's write the matrix $J$ for the system under study in terms of the effective potential: 
\begin{align}
\label{eq3.5}
\det J=\left|
    \begin{matrix}
	\frac{\partial^2 H}{\partial p_r^2} & \frac{\partial^2 H}{\partial p_r \partial r} \\
	\frac{\partial^2 H}{\partial p_r \partial r} & \frac{\partial^2 H}{\partial r^2} \\ 
    \end{matrix}
\right|=\left|
    \begin{matrix}
	1 & 0 \\
	0 & \frac{\partial^2 U_k}{\partial r^2} \\ 
    \end{matrix}
\right|=\frac{\partial^2 U_k}{\partial r^2}\Big|_{k=K}.
\end{align}	
The non-degeneracy of critical circles belonging to the family $\mathcal{C}_1^k$ (i.e., located in the pre-image of points of the curve $\gamma_1$) means, in according to~\eqref{eq3.5}, that $U''_{k(r)}(r)\neq 0$. This proves the equivalence of (a) and (b).

Let us compute this expression explicitly using~\eqref{eq3.4}, by substituting $k=k(r)=\Lambda(r)-f(r)\dfrac{\Lambda'(r)}{f'(r)}$ into $U''_{k}(r)$:
\begin{equation} \label{eq3.6}
U''_{k(r)} (r) =-\dfrac{1}{f^{2}(r)}\Lambda'(r)k'(r)=\dfrac{\Lambda'(r)}{f(r)}\left(\dfrac{\Lambda'(r)}{f'(r)}\right)'=\dfrac{\Lambda'(r)}{f(r)f'^2(r)}\left( f'(r)\Lambda''(r)-\Lambda'(r)f''(r) \right).
	\end{equation}
The chain of equalities~\eqref {eq3.6} proves the equivalence of (b)--(d) and the coincidence of the signs of the corresponding quantities.

Let's prove the second statement of the theorem.

The family $\mathcal{C}_1^\Lambda$ is investigated by the same technique that is used for the first family. Computing the function $U''_k(r)$ via~\eqref{eq3.4} and substituting $k=\Lambda(r)$ into this function gives
\begin{equation*}\frac{\partial^2 U_k}{\partial r^2}\big|_{k=\Lambda(r)}=\frac{\Lambda'^2(r)}{f^2(r)}.
\end{equation*}
The corresponding critical circles are of elliptical type, except for those points $r_i$ at which $\Lambda'(r_i)=0$.

Let's prove the third statement of the theorem.

Since a (semi-local) singularity is a local object, we will regard
all functions as functions defined on $\mathbb{R}^3\times S^1(\varphi)$. Let's fix an isoenergy surface $H=h$ in $\mathbb{R}^3(p_r,\,r,\,K)$ and consider the one-dimensional foliation of the plane $(p_r,\,r)$ by level lines $K=k$ of the additional first integral (with various levels $k$). From such a foliation, a two-dimensional foliation is obtained by multiplying with the circle $S^1(\varphi)$.

Let us establish the local form of the first integrals near the corresponding singularities.

Let $r^*$ be a simple zero of the function $f'\Lambda''-f''\Lambda'$, then (as we proved above) $k'(r^*)=0$ and $k''(r^*)\neq 0$. Let's show that $(h^*,\,k^*)=\gamma_1(r^*)=(U_{k(r^*)}(r^*),k(r^*))$ is a cusp point of the bifurcation curve $\gamma_1$. The following equalities are checked by direct calculation:
\begin{align*}
&\frac{\partial U_k(r)}{\partial r} \Bigr|_{k=k(r^*),\,r=r^*}=0,\quad
	\frac{\partial^2 U_k(r)}{\partial r^2} \Bigr|_{k=k(r^*),\,r=r^*}=0,\quad \\
&\frac{\partial^3 U_k(r)}{\partial r^3} \Bigr|_{k=k(r^*),\,r=r^*} = - \frac{k''(r^*) \Lambda'(r^*)} {f^2(r^*)} \neq 0,\quad 
	\frac{\partial^2 U_k(r)}{\partial k\, \partial r} \Bigr|_{k=k(r^*),\,r=r^*} = \dfrac{\Lambda'(r^*)}{f^2(r^*)}\neq 0.
\end{align*}
Due to results of Singularity theory (see, e.g., \cite[sect.~1.5, Whitney’s theorem, or sect.~8.1--8.3]{AVG:85}), in a small neighborhood of a point, there is a regular change of coordinates $p_1=p_r$, $p_2=p_2(K)$, $q_1=q_1(K,\,r)$, such that $p_2(k^*)=q_1(k^*,\,r^*)=0$ and
$2H=p_1^2+q_1^3-p_2 q_1+c(p_2)$, $K=K(p_2)$, where $c=c(p_2)$ is a smooth function, $K=K(p_2)$ is a function inverse to the function $p_2=p_2(K)$.

In the new local coordinates, the equation $U_k'(r)=0$ is rewritten as $3q_1^2-p_2=0$, and the bifurcation curve $\gamma_1$ near this point is parametrized as $(-q_1^3+c(3q_1^2)/2,K(3q_1^2))$. This means that the point $(h^*,\,k^*)$ is an ordinary cusp point (having a semi-cubic parabola type) of the curve $\gamma_1$. In the phase space, it corresponds to a parabolic critical circle $\mathcal O_{r^*, k^*}$ (and the singular integral manifold containing this circle is a ``cuspidal torus'' if there are no other singular circles on it).

To prove the second part of the third statement, take a value $r^*$ such that $0<r^*<L$ and $\Lambda'(r^*)=0$ (on the bifurcation diagram in the $(h,k)$-plane, this is a tangency point $(0,\,\Lambda(r^*))$ of the curve $\gamma_1$ and the axis $Ok$, see e.g.\ the point $\gamma_1(r_1^*)$ in Fig.~\ref{fig_10},~\textit{b}). The following relations are obviously true for $k=\Lambda(r^*)$, $r=r^*$:
\begin{align*}
\frac{\partial }{\partial r}\left(\frac{k-\Lambda(r)}{f(r)}\right)=0, \quad
		\frac{\partial }{\partial k}\left(\frac{k-\Lambda(r)}{f(r)}\right)=\frac{1}{f(r^*)} \ne 0, \\
		\frac{\partial^2 }{\partial r^2}\left(\frac{k-\Lambda(r)}{f(r)}\right)=-\frac{\Lambda''(r^*)}{f(r^*)} \ne 0.
	\end{align*}
In a neighbourhood of this point, due to a parametric version of the Morse lemma, there is a regular change of coordinates $p_1=p_r$, $p_2=p_2(K)$, $q_1=q_1(K,\,r)$ (generally speaking, different from the previous one) such that
 \begin{equation*}
p_2(\Lambda(r^*))=q_1(\Lambda(r^*),\,r^*)=0, \qquad 
    \dfrac{K-\Lambda(r)}{f(r)}=\pm(q_1^2 - p_2),
 \end{equation*}
therefore $2H=p_1^2+(q_1^2-p_2)^2 = p_1^2 + q_1^4 - 2p_2q_1^2 + p_2^{2}$, $K=K(p_2)$, where $K(p_2)$ is the inverse function of the function $p_2(K)$.

Consequently, in a small neighbourhood of this singularity, the Liouville foliation has the structure of an ``elliptic fork'' (see, for example, \cite {BRF:2000}).

The theorem is proved.
\end{proof}

\begin{remark}
Usually the term ``elliptic fork'' from Remark~\ref{rem3.6} is used for singularities in which an involution $(p_1,p_2,q_1,q_2) \mapsto (-p_1,p_2,-q_1,q_2)$ preserving the first integrals of the system is a symplectomorphism. In our case, this is generally not the case, i.e., in this sense, the fork mentioned in Theorem~\ref{th3.7} is ``asymmetric''. 
\end{remark}	
	
\begin{remark}
Since, according to Theorem~\ref{th3.3} points of rank $0$ have the type ``center-center'', it is not difficult to establish the topology of the Liouville foliation in their neighborhood. To do this, we apply the Eliasson--Wey theorem, according to which, a neighborhood of a rank-0 singular point of the center-center type 
is fiber-wise symplectomorphic to the direct product $A\times A$ (here an elliptic singularity of a Hamiltonian system with 1 degree of freedom is denoted by $A$).
\end{remark}

\begin{remark} \label{rem3.11}
Equivalence of the conditions (a)--(c) and description of the type of critical circle $\mathcal O_{r,k(r)}$ in terms of the signs of the quantities $U_k''(r)$ and $-k'(r)\Lambda'(r)$ from the first item of Theorem~\ref {th3.7}, as well as the second item of Theorem~\ref {th3.7} follow from \cite[Proposition~2, items~(A) and (B)(c)]{KO:20}. The criterion of parabolicity of the critical circle $\mathcal O_{r,k(r)}$ in terms of the effective potential (respectively, the function $k(r)$) follows from \cite[Proposition~2~(A)]{KO:20} (respectively \cite[Proposition~2~(B)(c,e)]{KO:20}).
\end{remark}

\begin{corollary} \label {cor3.12}
The additional first integral $K\big|_{\{H=h\}}$ is a Bott function for values $h>0$ different from the critical values of the function $H$ and from the abscissa of the cusp points of the curve $\gamma_1$. 
\end{corollary}

\begin{proof}
It is easy to show that the additional first integral $K\big|_{\{H=h\}}$ is Bott on a non-singular isoenergy manifold $Q^3_h=\{H=h\}$ if and only if all singular points of rank 1 of the momentum map contained in $Q^3_h$ are non-degenerate (i.e., elliptic or hyperbolic). In according to the first and second statements of Theorem~\ref{th3.7}, all critical circles in $Q^3_h$ are non-degenerate if and only if $(f'(r)\Lambda''(r)-f''(r)\Lambda'(r))\Lambda'(r)\ne0$ at all orbits $\mathcal O_{r,k(r)}\subset Q^3_h$. The latter inequality is satisfied, since according to the third statement of Theorem~\ref {th3.7}, the condition $f'\Lambda''-f''\Lambda'=0$ describes the cusp points of the curve $\gamma_1$, and $\Lambda'(r)=0$ is prohibited by the condition $h>0$.
\end{proof}

\subsection{Semi-global 4D singularities}

In this section, we will assume that the pair of functions $f,\Lambda$ satisfies some additional conditions, so-called ``genericity in a strong sense''.

Denote $I_{hyp}=\{r\in[0,L]\setminus\{r_i\}\mid(f'(r)\Lambda''(r)-f''(r)\Lambda'(r)) \Lambda'(r)<0\}$ (see footnote$^{\ref{footnote1}}$). This is a disjoint union of intervals with endpoints $r_i,r_j^*,r_\ell^\circ$. According to item~1 of Theorem~\ref{th3.7} (or item~2 of Theorem~\ref{th5.13}), these intervals correspond to 1-parameter families of hyperbolic critical circles $\mathcal O_{r,k(r)}$, $r\in I_{hyp}$, and the endpoints of the intervals other than $r_i$ are degenerate critical circles $\mathcal O_{r_j^*,k(r_j^*)}$, $\mathcal O_{r_\ell^\circ,k(r_\ell^\circ)}$.

Let $\Gamma_-$ be the curve obtained from the curve $\Gamma=(f,\Lambda)$ by reflection with respect to the line $\{f=0\}$, i.e., $\Gamma_-(r)=(-f(-r),\Lambda(-r))$, $r\in[-L,0]$.

\begin{assumption} \label {assump3.12}
Let us assume that, in addition to the conditions 1--6 from Assumption~\ref{assump2.4}, the pair of functions $\Gamma=(f,\Lambda)$ satisfies the following conditions:
\begin{itemize}
\item[7.] all self-intersection points of the curve $\Gamma|_{I_{hyp}}$ are transversal;
\item[8.] any straight line tangent to the curve $\Gamma|_{I_{hyp}}\cup\Gamma_-|_{-I_{hyp}}$ at two points is not tangent to it at any third point;
\item [9.] the tangent line to the curve $\Gamma$ at an inflection point $\Gamma(r_{\ell_1}^\circ)$ is not tangent to the curve $\Gamma|_{I_{hyp}\cup\{r_\ell^\circ\}}\cup\Gamma_-|_{-I_{hyp}\cup\{-r_\ell^\circ\}}$ at any other point.
\end{itemize}
We will say that the pair of functions $f,\Lambda$ is {\em  generic in a strong sense} if all the conditions 1--9 are satisfied.
\end{assumption}

Let us describe the topology of the Liouville foliation in a neighbourhood of singular fibers $\{H=h_*,\, K=k_*\}$ (i.e., {\em semi-global singularities} of the Liouville foliation).

According to Theorem~\ref {th3.7}, any inflection point of the curve $(f,\Lambda)$ corresponds to a rank-1 critical circle of parabolic type. The fiber containing it is a cuspidal torus, see Fig.~\ref {fig_2}, \textit{a}. When the energy level $h$ passes through $h_*$, a pair of 3-atoms of types $A$ and $B$ is born or destroyed in the Fomenko molecule of the Liouville foliation on $Q^3_h$ (Definition~\ref{def4.1}). 

\begin{figure}[ht!] 
    \begin{minipage}[r]{0.49\linewidth}
		\begin{center}
            \includegraphics[scale=.3]{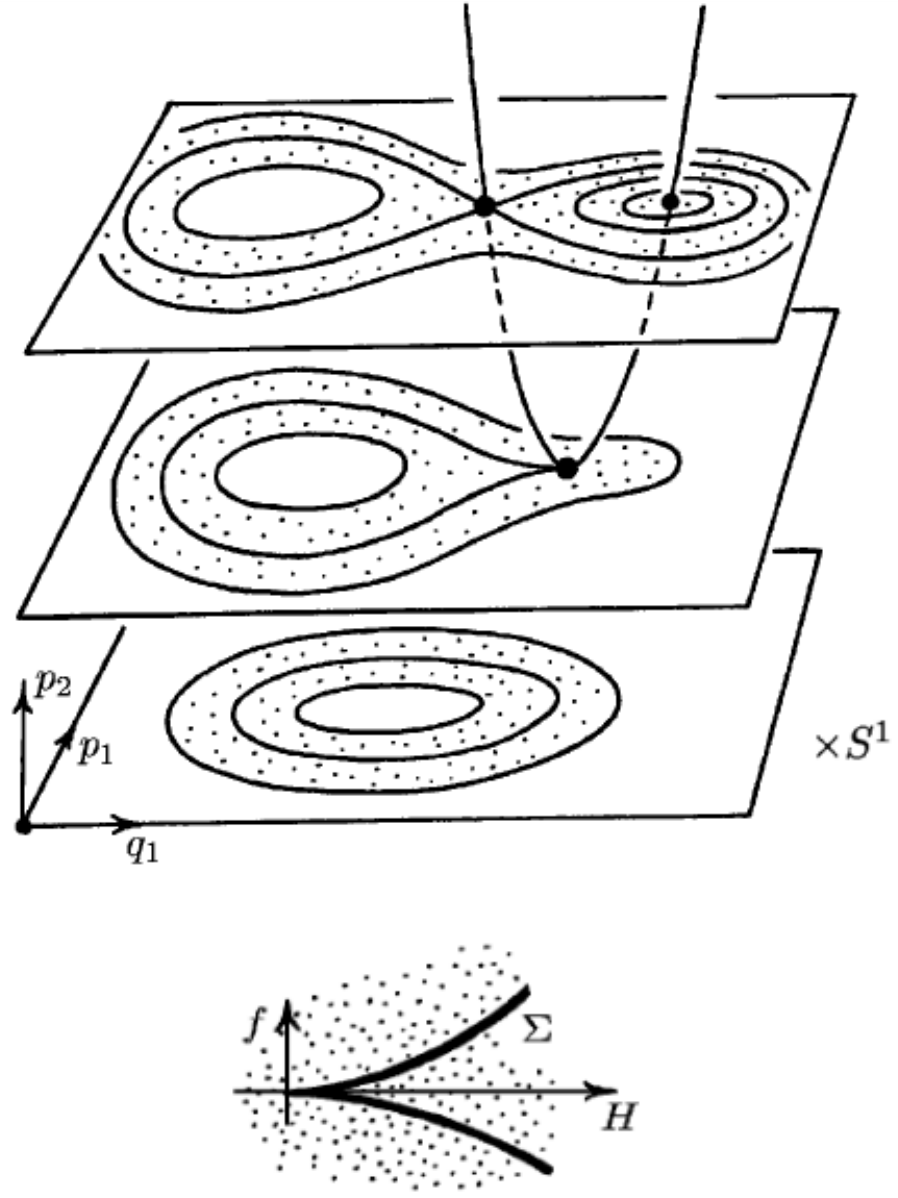} \\ \textit{a)}
		\end{center}
    \end{minipage} 
    \begin{minipage}[r]{0.49\linewidth}
		\begin{center}
            \includegraphics[scale=.28]{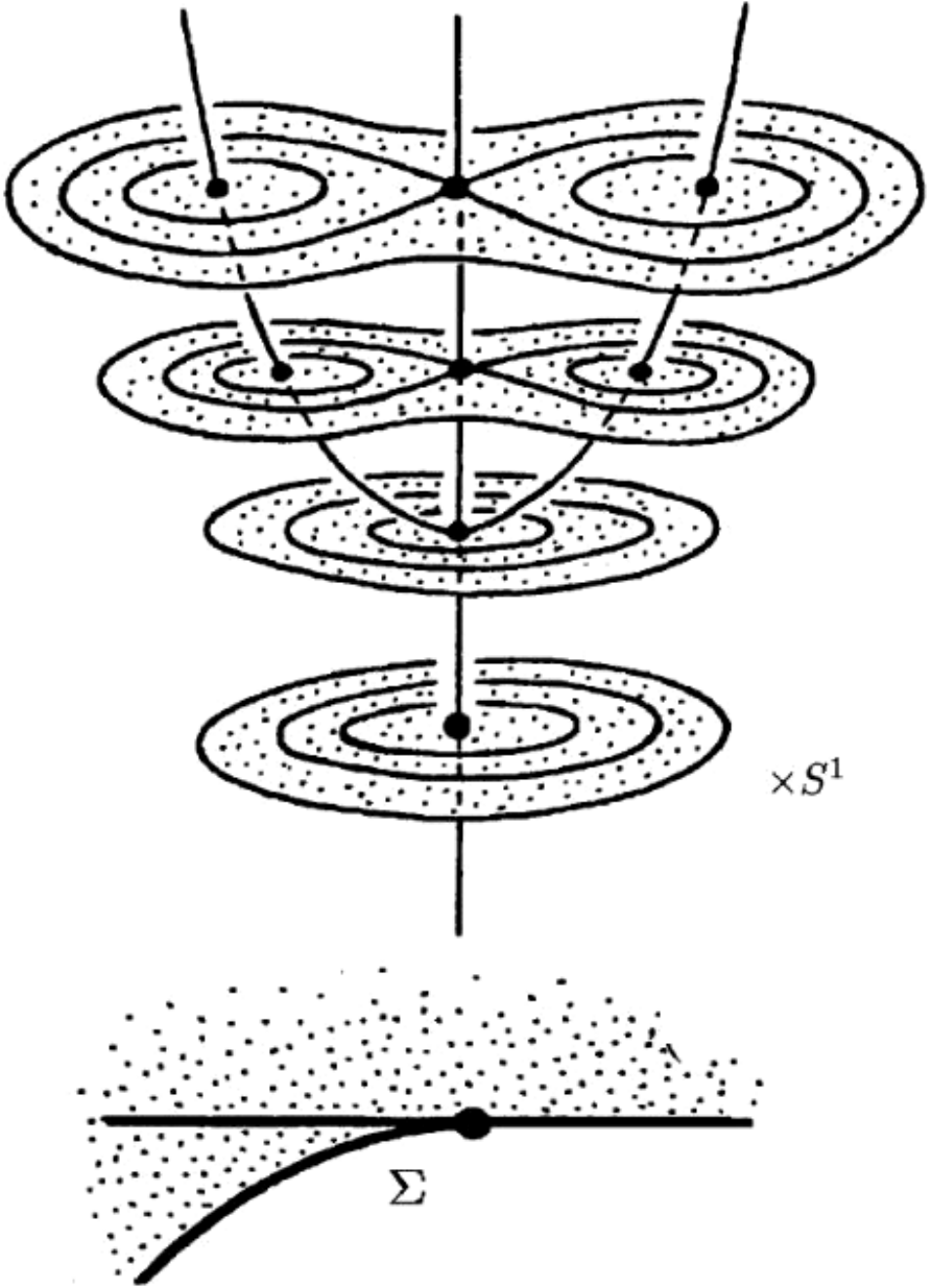} \\ \textit{b)}
        \end{center}
    \end{minipage} 
    \caption{Liouville foliation near degenerate singular fibers of the magnetic geodesic flow $S(f,\Lambda)$:
    \textit{a)} near a cuspidal torus, \textit{b)} near an elliptic fork}
    \label {fig_2}
\end{figure}

According to Theorem~\ref {th3.7}, any point of the curve $(f,\Lambda)$ with a horizontal tangent corresponds to a rank-1 critical circle having the singularity type ``asymmetric elliptic fork'', see Fig.~\ref {fig_2}, \textit{b}.

Since we want to study singular fibers, rather than just singular orbits, we need (unlike Theorem~\ref{th3.7}) to look not only at the inflection points of the curve $\Gamma=(f,\Lambda)$ and its points with a horizontal tangent, but also at pairs of points with a common tangent (which correspond to self-intersection points $(h_*,k_*)$ of the bifurcation curve $\gamma_1$), see Fig.~\ref {fig_3} (\textit{a}, \textit{c}, \textit{e}), where the curve $\gamma=(a,k)=(\pm\sqrt{2h},k)$ is obtained from the bifurcation curve $\gamma_1=(h,k)$ by some choice of sign (see \eqref{eq5.3} below).
Such pairs of points correspond to a ``splitting hyperbolic singularity of rank 1 and complexity 2'' (see~\cite[rem.~4]{KO:20}). 
As the energy level $h$ passes through $h_*$, the topology of the Liouville foliation on the isoenergy manifold $Q^3_h$ changes\footnote{This property means that the Liouville foliation is topologically unstable near such fibers~\cite{IGS}.} (near a singular fiber $k=k_*$ in $Q^3_{h_*}$) as shown in Fig.~\ref{fig_3} (\textit{b}, \textit{d}, \textit{f}), with the appearance of the 3-atom $D_1\approx V_{++}\times S^1\approx V_{--}\times S^1$ or $D_2 \approx V_{+-}\times S^1\approx V_{-+}\times S^1$ in $Q^3_{h_*}$.

\begin{figure}[ht!]
    \begin{minipage}[r]{0.57\linewidth}
		\centering
        \includegraphics[scale=1.1]{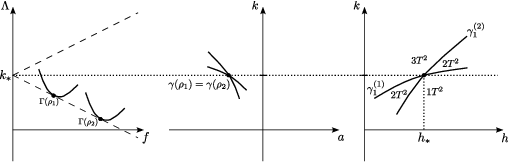} \\ \textit{a)}
    \end{minipage} 
    \hfill
    \begin{minipage}[r]{0.43\linewidth}
		\centering
        \includegraphics[scale=1.35]{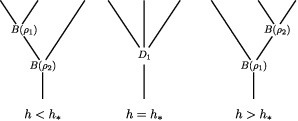} \\ \textit{b)}
    \end{minipage}
    \vfill
    \begin{minipage}[r]{0.57\linewidth}
		\centering
        \includegraphics[scale=1.1]{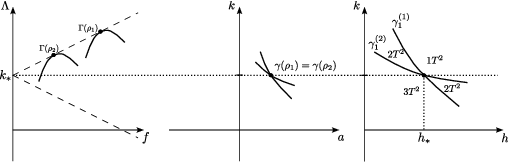} \\ \textit{c)}
    \end{minipage} 
    \hfill
    \begin{minipage}[r]{0.43\linewidth}
		\centering
        \includegraphics[scale=1.35]{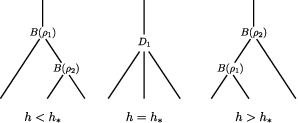}\\ \textit{d)}
    \end{minipage}
    \vfill
    \begin{minipage}[r]{0.57\linewidth}
		\centering
        \includegraphics[scale=1.1]{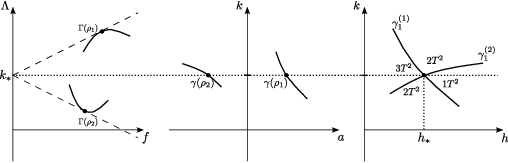} \\ \textit{e)}
    \end{minipage} 
    \hfill
    \begin{minipage}[r]{0.43\linewidth}
	\centering
        \includegraphics[scale=1.35]{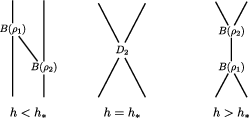} \\ \textit{f)}
    \end{minipage}
    \caption{The curves $\Gamma$, $\gamma$ and $\gamma_1$ near the saddle 3-atom $V_{\sigma_1\sigma_2}\times S^1$ (\textit{a}, \textit{c}, \textit{e}) and the corresponding bifurcation 
    of the Fomenko molecule (\textit{b}, \textit{d}, \textit{f}), where $(\sigma_1,\sigma_2)=(+,+)$, $(-,-)$ and $(-,+)$ respectively, with notation $\gamma^{(1)}_1=\gamma_1\big|_{(\rho_1-\varepsilon,\rho_1+\varepsilon)}$, $\gamma^{(2)}_1=\gamma_1\big|_{(\rho_2-\varepsilon,\rho_2+\varepsilon)}$ }
    \label {fig_3}
\end{figure}

\section{Topology of the Liouville foliation on 3D isoenergy manifolds of the system $S(f,\lambda)$}

By abusing notations, in this section (except for Remark~\ref{rem4.6}), the symbols $h$ and $k$ will denote some fixed constants that are the values of the first integrals $H$ and $K$ (in the previous sections, the notations $h$, $k$ were used for a pair of functions defining the bifurcation curve $\gamma_1 =(h,k)$ from Theorem~\ref{th3.3}).

\subsection{Invariants of rough Liouville equivalence}

Let's define an important concept used below.

\begin{definition} \label {def4.1}
Let $W$ be the Reeb graph of the function $K\big|_{Q^3_h}$. The {\em Fomenko molecule} of the Liouville foliation on $Q^3_h$ (and of the function $K\big|_{Q^3_h}$) is the graph $W$, each vertex of which is assigned with a 3-atom together with the corresponding bijection between the set of its boundary tori and the set of all edges incident to this vertex.
\end{definition}

It is well known that the Fomenko molecule is an invariant of the Liouville equivalence of a system on the isoenergy manifold\footnote{Let's say that two Liouville foliations on $Q^3_h$ and $\tilde Q^3_{\tilde h}$ have the same topology (and the corresponding systems on $Q^3_h$ and $\tilde Q^3_{\tilde h}$ are {\em Liouville equivalent}) if there is an orientation-preserving diffeomorphism $\Phi: Q^3_h\to \tilde Q^3_{\tilde h}$ that transforms fibers of the first foliation into fibers of the second foliation and preserves the orientation of critical circles. The orientation on the isoenergy manifold is determined by the formula~\eqref{eq4.2}.\label{footnote3}}.
This invariant is incomplete. It is sometimes called {\em rough}, in order to distinguish it from the complete invariant of the Liouville equivalence of the system on $Q^3_h$ called the Fomenko--Zieschang {\em marked molecule}.

\begin{theorem} \label{th4.2}
For a magnetic geodesic flow $S(f,\Lambda)$ on the sphere, and for an energy level $h\in\mathbb R$:
\begin{enumerate}
\item The isoenergy manifold $Q^3_h=\{H=h\}\subset T^*M$ is regular if and only if $h\neq 0$. It is diffeomorphic to $\mathbb{R}P^3$ for any $h>0$.
\item If the pair of functions $f,\Lambda$ satisfies the conditions 1--6 from Assumption~\ref {assump2.4} and the energy level $h>0$ is different from the abscissa of the cusp points of the curve $\gamma_1$ and of the intersection points of this curve with the lines $\{k=\Lambda(0)\}$ and $\{k=\Lambda(L)\}$, then $K\big|_{Q^3=h}$ is a Bott function and the Fomenko molecule of this function is a tree with 3-atoms $A$ at the terminal vertices and saddle 3-atoms $V=V_{\sigma_1\dots\sigma_m} \times S^1$ (described in Remark~\ref{rem4.3}, see below) at other vertices. 
On the bifurcation diagram, such a saddle atom $V_{\sigma_1 \dots \sigma_m} \times S^1$ corresponds to a point through which $m$ arcs of the curve $\gamma_1$ pass, of the form $\gamma_1 (\rho_1)=\dots=\gamma_1 (\rho_m)$, with $0<\rho_1 <\dots<\rho_m <L$, and the cotangents of the inclination angles of the vectors $\gamma_1'(\rho_i)$ are equal to $\frac{\partial H}{\partial K}\big|_{\mathcal O_{\rho_i, k(\rho_i)}}$ and have the signs $\sigma_i$, $i=1,\dots,m$. In particular, in the case of $\sigma_i =+1$, the orientation of the critical circle $\mathcal O_{\rho_i, k(\rho_i)}$ by the phase flow of the system coincides with the orientation given by the Hamiltonian $S^1$ action generated by the function $K$, and in the case of $\sigma_i = -1$ is the opposite of it.
\end{enumerate}
\end{theorem}

\begin{remark} \label{rem4.3} 
(a) The saddle 3-atom $V=V_{\sigma_1 \dots \sigma_m}\times S^1$ is a direct product of the corresponding 2-atom $V_{\sigma_1\dots \sigma_m}$ (introduced in \cite{KO:20}) by the circle $\varphi$. The singular fiber of this 2-atom is a chain of $m+1$ circles
(i.e., a graph of the form $\subset\!>\!\!\!\bullet\!\!\!<\!>\!\!\!\bullet\!\!\!<\dots>\!\!\!\bullet\!\!\!<\!>\!\!\!\bullet\!\!\!<\!\supset$). Let us give a more precise description of this 2-atom \cite[Lemma~3] {KO:20}:
\begin{itemize}
\item We will represent the 2-atom as a neighborhood of a critical level of a Morse function on a surface.
Consider a surface embedded in $\mathbb R^3 (p_r,r,K)$ with the function $K$ on it (a height function). We will assume that it is a Morse function. Consider a saddle critical level $K=K_{crit}$ of the function $K$ on this  surface (this is a planar graph of valence 4, the projection of the surface to the plane $(p_r, r)$ is a local diffeomorphism near any of its vertices). By coloring the areas $K<K_{crit}$ and $K>K_{crit}$ in black and white, respectively, we obtain a checkerboard coloring of the surface near the graph $K=K_{crit}$.
\item 
It is known that there are exactly two ways of chess colouring in a small neighbourhood of a vertex of degree 4 of a planar graph (Fig.~\ref{fig_4},~\textit{a}; they are transformed into each other by permutation of colours). We will call such a local phase portrait a {\em cross}. To distinguish these two types from each other, we will mark them with the sign $\sigma=\pm1$.
\item
Next, let's take $m$ crosses marked with the signs $\sigma_1, \dots, \sigma_m$, and form from them a pre-image of a neighborhood of the critical value of a function with $m$ Morse singularities, so that 
neighboring crosses are glued together, and the remaining free ends of the first and last crosses are glued to each other (if the signs assigned to two neighboring crosses are different, then one should make a twist when gluing their corresponding ends; see Fig.~\ref{fig_4}, \textit{b}).
Examples of 2-atoms obtained as a result of such gluing are shown in Fig.~\ref{fig_5}.
Next, in the notation of atoms $V_{\sigma_1 \dots \sigma_m}$ for brevity, we will write $V_{\pm\dots\pm}$ instead of $V_{\pm1\dots\pm 1}$.
\end{itemize}

(b) It is obvious that the 2-atoms $V_{\sigma_1\dots\sigma_m}$ and $V_{\sigma_m\dots\sigma_1}$ are topologically the same in the following sense: Liouville foliations on 3-atoms $V_{\sigma_1 \dots \sigma_m} \times S^1$ and $V_{\sigma_m \dots\sigma_1} \times S^1$ have the same topology (see footnote$^{\ref{footnote3}}$). It is also obvious that if all signs in the notation of a 2-atom $V_{\sigma_1\dots\sigma_m}$ are replaced with opposite ones, then topologically it will be the same atom, but with the opposite direction of the additional first integral. Changing the orientation on the 2-atom $V_{\sigma_1\dots\sigma_m}$ also does not change the topology of the Liouville foliation on the corresponding 3-atom (this is obvious if one looks at the fiberwise diffeomorphism $(p_r,K,r,\varphi) \to (-p_r,K,r,-\varphi)$ that reverses orientation of 2-atoms).

(c) If the pair of functions $f,\Lambda$ satisfies the additional condition of ``genericity in a strong sense'' (see condition 8 in Assumption~\ref{assump3.12}),
then only the following types of saddle 3-atoms can appear in the Fomenko molecule of the Bott function $K|_{Q^3_h}$ for the magnetic geodesic flow $S(f,\Lambda)$ on the sphere:
$V_\sigma\times S^1$ and $V_{\sigma_1\sigma_2}\times S^1$, $\sigma,\sigma_1,\sigma_2 \in\{+1,-1\}$ (topologically arranged as 3-atoms $B$, $D_1$ and $D_2$, see Fig.~\ref {fig_5}).
\end{remark}

Let's introduce a definition that will be actively used in proving Theorem~\ref{th4.2}.
\begin{definition}
A {\em domain of possible motion} $R_{h,k}$ is the projection of the integral manifold corresponding to the values $H=h=\const$, $K=k=\const$ into the configuration manifold $M$.
\end{definition}

\begin{remark}\label{rem:Delta} 
This definition implies that, under the conditions of the problem, the domain of possible motion on the sphere $M=S^2$ is given as follows:
\begin{equation*}
	R_{h,k}=\left\{(r,\varphi): U_k(r) \leqslant h \right\}.
	\end{equation*}
These conditions can be written as
	\begin{equation*}
	U_k(r) \leqslant h \quad \iff \quad \Lambda(r) - \sqrt{2h}f(r) \leqslant k \leqslant \Lambda(r) + \sqrt{2h}f(r).
	\end{equation*}
By introducing the functions $g_ -(r)=\Lambda(r) - \sqrt{2h}f(r)$, $g_+(r)=\Lambda(r) +\sqrt{2h}f(r)$, by definition we get that the domain of possible motion has the form 
	\begin{equation*}
		R_{h,k}=\left\{(r,\varphi): g_-(r)\leqslant k \leqslant g_+(r) \right\}.
	\end{equation*}
The projection of this set to the $r$ coordinate is the disjoint union of segments $\Delta_i\subset\left[0, L\right]$.
Any such a segment $\Delta_i$ corresponds to the following subset of the sphere $M=S^2$:
\begin{itemize}
\item a disk if either $0\in\Delta_i$ and $L\not\in\Delta_i$, or $L\in\Delta_i$ and $0\not\in\Delta_i$ (this is possible only in the case of $k\in\{\Lambda(0),\ \Lambda(L)\}$, since $g_-(0)=g_+(0)=\Lambda(0)$ and $g_-(L)=g_+(L)=\Lambda(L)$);
\item an annulus if $0\notin\Delta_i$, $L\notin\Delta_i$ (this is true for any $k\in\mathbb R\setminus\{\Lambda(0),\Lambda(L)\})$;
\item the whole $S^2$ if $\Delta_i =\left[0,\,L\right]$ (this is true if and only if
$k=\Lambda(0)=\Lambda(L)$, $h\geqslant\sup\limits_{0<r<L}U_k (r)$).
	\end{itemize}
On a planar curve $\Gamma(r)=(f(r),\Lambda(r))$, such segments $\Delta_i$ correspond to the pieces of intersection of this curve with the sector $\{(f,\Lambda): |\Lambda-k|\leqslant\sqrt{2H}f\}$.
\end{remark}

\begin{figure}[ht!]
    \centering
    \begin{minipage}{0.49\linewidth}
        \centering
	  \includegraphics[scale=1,angle=0]{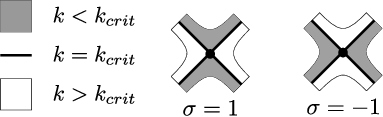}  \\ \textit{a)}
    \end{minipage}
    \hfill
    \begin{minipage}{0.49\linewidth}
        \centering
        \includegraphics[scale=.8,angle=0]{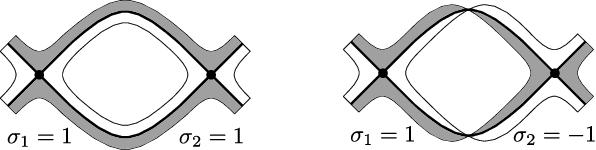} \\ \textit{b)}
    \end{minipage}
    \caption{\centering \textit{a)} Definition of a marked cross, \textit{b)} the rules for gluing crosses}
    \label{fig_4}
\end{figure}

\begin{figure}[ht!]
    \centering
    \begin{minipage}{0.49\linewidth}
        \centering
		\includegraphics[scale=1,angle=0]{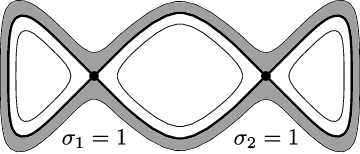} \\ \textit{a)}
    \end{minipage}
    \hfill
    \begin{minipage}{0.49\linewidth}
        \centering
		\includegraphics[scale=1,angle=0]{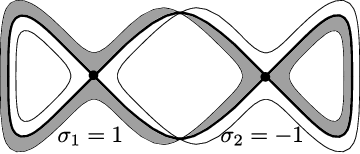} \\ \textit{b)}
    \end{minipage}
    \caption{\centering Examples of atoms glued from crosses: \textit{a)} $V_{++}=D_1$, \textit{b)} $V_{+-}=D_2$ }
    \label{fig_5}
\end{figure}

In order to prove Theorem \ref {th4.2}, let's pose the question: how, using the given functions $f$, $\Lambda$, to construct the Fomenko molecule for the Liouville foliation given by the first integral $K$ on $Q^3_h$, for a given energy level $h$?

\begin{lemma} \label {le4.5} 
Suppose that $h>0$ is a non-singular energy level and $K|_{Q^3_h}$ is a Bott function. Then the functions $g_+(r)$ and $g_-(r)$ on $[0,L]$ are Morse functions. Let $k\in\mathbb R\setminus\{\Lambda(0),\Lambda(L)\}$. Then:

\begin{enumerate}
\item if $k$ is a local maximum of the function $g_+(r)$ or a local minimum of the function $g_-(r)$, then the corresponding extremum points are in one-to-one correspondence with the 3-atoms $A$ at the vertices of the Fomenko molecule;
\item if $k$ is a local minimum of the function $g_+(r)$ and/or a local maximum of the function $g_-(r)$, then the corresponding extremum points are in one-to-one correspondence with the critical circles of a saddle 3-atom $V_{\sigma_1\dots\sigma_m} \times S^1$ of the Fomenko molecule. In this case, $\sigma_1,\dots, \sigma_m$ is a set of signs corresponding to the points of local extremes $\rho_1,\dots,\rho_m$ of the functions $g_\pm(r)$ with the extreme value $k$, ordered in ascending order, furthermore $\sigma_i=+$ if $\rho_i$ is a point of a local minimum of the function $g_+(r)$, and $\sigma_i=-$ if $\rho_i$ is a point of a local maximum of the function $g_-(r)$.
\end{enumerate}
\end{lemma}

\begin{proof}
The statements of the lemma follow from \cite[Lemmas~2 and~3]{KO:20}. Let's give its proof for completeness.

As noted earlier, the domain of possible motion can be defined as $R_{h,k}=\left\{(r,\varphi): g_-(r)\leqslant k\leqslant g_+(r)\right\}$. By fixing a non-singular energy level $h>0$ and allowing the number $k$ to run through all possible real values, we obtain a flat region 
\begin{equation} \label{eq4.1} 
 \Pi_h=\{(k,r)\in\mathbb R\times(0,L) : g_-(r) \leqslant k\leqslant g_+(r)\}.
\end{equation}
Obviously, the connected components of the sets $\Pi_h\cap\{k=\const\}=\{k\}\times R_{h,k}$ and the fibers of the Liouville foliation of the system are in one-to-one correspondence, therefore the Reeb graph of the function $K$ (on a nonsingular $Q^3_h$) coincides with the Reeb graph of the function $k$ restricted to the flat domain $\Pi_h$, i.e., it can be constructed as follows: declare each connected component of the set $\Pi_h\cap\{k=\const\}$ as a separate point and introduce the quotient topology on the resulting space.

Let us explicitly construct the Liouville foliation of the system on $Q^3_h$ with the exception of the fibers $Q^3_h\cap\{K=\Lambda(0)\}$ and $Q^3_h\cap\{K=\Lambda(L)\}$.

The isoenergy manifold $Q^3_h\cap T^*(M\setminus\{N,S\})$ is given by the equation $\dfrac{p_r^2}{2}+\dfrac{(K-\Lambda(r))^2}{2f^2(r)}=h$ in the phase space $T^*(M\setminus\{N,S\})$ with coordinates $(p_r,K,r,\varphi)$. Since the coordinate $\varphi$ is not explicitly included in this equation, $Q^3_h$ is a direct product of a 2-dimensional surface (given by the same equation) in the 3-dimensional space $\mathbb R^2\times(0,L)$ by a circle. Let's take $\varphi=\const$ as a cross-section, and choose $p_r$ and $r$ as local coordinates on it. Such a cross-section $\{K=k_\pm(p_r,r),\ \varphi=\const\}$ can also be visualized as the union of ellipses (defined by the same equation with a fixed value $r$) in the planes $\{r=\const\}$ with coordinates $p_r,K$. 
We get a presentation of the isoenergy manifold $Q^3_h\cap T^*(M\setminus\{N,S\})$ in the form $K=k_\pm(p_r,r)$, i.e., as a direct product of a circle and the union of the graphs of the following two functions:
\begin{equation*}
   k_+(p_r,r)=\Lambda(r) + \sqrt{2h-p_r^2}\,f(r), \quad k_-(p_r,r)=\Lambda(r) - \sqrt{2h-p_r^2}\,f(r).
\end{equation*}
Fibers of the Seifert foliation have the form $p_r=\const$, $K=\const$, $r=\const$, $\varphi\in\mathbb{R}/2\pi\mathbb{Z}$, thus the Liouville foliation is constructed as the result of multiplication of a one-dimensional foliation (on the cross-section $\Sigma=\{K=k_\pm(p_r,r),\ \varphi=\const\}$, see Fig.~\ref {fig_6}) with the circle $\varphi$.

\begin{figure}[ht!] 
    \centering
    \includegraphics[scale=2]{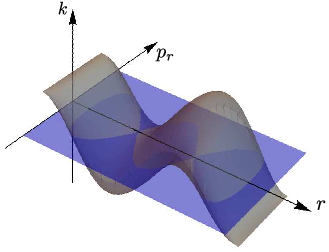}
    \caption{Example of a cross-section surface $\Sigma\subset Q^3_h\cap\{\varphi=\varphi_0\}\subset\mathbb R^2\times(0,L)\times\{\varphi_0\}$ and a saddle level of the Morse function $K|_{\Sigma}$ corresponding to the 2-atom $V_{+-}=D_2$}
    \label{fig_6}
\end{figure}

We need to study bifurcations of level surfaces of the Bott function $K$ on $Q^3_h\setminus(\{K=\Lambda(0)\}\cup\{K=\Lambda(L)\})$.
Since $Q^3_h\setminus(\{K=\Lambda(0)\}\cup\{K=\Lambda(L)\})\subseteq Q^3_h \cap T^*(M\setminus\{N,S\})$, it suffices to study bifurcations of level lines of the Morse function $K$ on the cross-section $\{K = k_\pm(p_r,r),\ \varphi=\const \}\setminus(\{K=\Lambda(0)\}\cup\{K=\Lambda(L)\})$. A level line of the function $K$ in this cross-section can be either a point or a circle or a chain of circles. It is obvious that the topology of level lines changes at critical points of the functions $k_\pm(p_r,r)$ and only at them:
\begin{equation*}
\begin{cases}
\dfrac{\partial k_\pm(p_r,r)}{\partial p_r}=0, \\
\dfrac{\partial k_\pm(p_r,r)}{\partial r}=0,
\end{cases}
\iff \quad 
\begin{cases}
p_r=0, \\
\Lambda'(r)\pm\sqrt{2h}f'(r)=0.
\end{cases}.
\end{equation*}
To determine the types of critical points, we calculate the Hesse matrix at them:
\begin{equation*}
d^2 k_\pm=
\begin{pmatrix}
\mp\dfrac{f(r)}{\sqrt{2h}} & 0 \\
0 & \Lambda''(r)\pm\sqrt{2h}f''(r) \\
\end{pmatrix}.
\end{equation*}

\begin{table}[htbp]
 \caption{Calculation of types and signs of critical circles}
 \center{
 \begin{tabular}{|c|c|c|c|c|}
 \hline
 № &\specialcell{Type of critical \\ points $r$} & \specialcell{ Type of critical \\ points $(p_r=0,r)$} & \specialcell{Type of critical \\circles $\mathcal O_{r,k(r)}$} & \specialcell{Sign of the critical \\ circle $\mathcal O_{r,k(r)}$} \\
\hline
 1 & local max of $g_+(r)$ & local max of $k_+(p_r,r)$ & elliptic & $\sigma=1$\\
\hline
 2 & local min of $g_+(r)$ & saddle of $k_+(p_r,r)$ & hyperbolic& $\sigma=1$\\
 \hline
 3 & local max of $g_-(r)$ & saddle of $k_-(p_r,r)$ & hyperbolic & $\sigma=-1$\\
 \hline
 4 & local min of $g_-(r)$ & local min of $k_-(p_r,r)$ & elliptic & $\sigma=-1$ \\
			\hline
		\end{tabular}
	}
    \label{table1}
\end{table}
The results are shown in Table~\ref{table1}.

Singularities 2 and 3 in Table~\ref {table1} are, in fact, crosses in the sense of the definition given in Remark~\ref {rem4.3}. Constructing a picture of level lines of the Morse function $K$ on the surface $\{K=k_\pm(p_r,r),\ \varphi=\const\}$ is equivalent to gluing crosses described in that definition, i.e., also leads to a 2-atom $V_{\sigma_1\dots \sigma_m}$. Similarly, singularities 1 and 4 correspond to a 2-atom $A$.

The lemma is proved.
\end{proof}

\begin{remark}[to Table~\ref {table1} and Figures~\ref{fig_4} and~\ref{fig_5}] \label{rem4.6}
(a) Figures~\ref{fig_4} and~\ref{fig_5} show the projection to the plane $(p_r,r)$ of the picture of level lines $\{K\big|_\Sigma=k\}$ of the Morse function $K$ on the surface $\Sigma=\{K=k_\pm(p_r,r),\ \varphi=\const\} \subseteq\{H=h_*,\ \varphi=\const\}$ near the critical value $k_*=k(\rho_1)=\dots=k(\rho_m)$.
Here and everywhere else in this remark, $h(r)$ and $k(r)$
denote the functions from Theorem~\ref{th3.3}, 
while $h$, $h_*$, $k$ and $k_*$ 
denote constants that are values of the functions $H$ and $K$. The sign $\sigma_i=\pm1$ corresponding to the critical circle $\mathcal O_{\rho_i,k(\rho_i)}$ of the given 3-atom $V=V_{\sigma_1\dots\sigma_m}\times S^1\subset Q^3_{h_*}$ coincides with the direction of the external normal (along the $Ok$ axis) to the flat region $\Pi_h$ at the point $(k_*,\rho_i)$.
Therefore, this sign coincides with the sign of the value $h'(\rho_i)/k'(\rho_i)$.
Thus, the vectors $\sgrad H$ and $\sgrad K$ at the critical circle $\mathcal O_{\rho_i,k(\rho_i)}$ are proportional to each other with a proportionality coefficient having the sign $\sigma_i$. See~\cite[Lemma~3]{KO:20}. Let's show this explicitly: since $k_*=\Lambda(\rho_i) + \sigma_i \sqrt{2h_*} f(\rho_i)$, where $\sigma_i=\pm1$, then $\sigma_i =\frac{k_*-\Lambda(\rho_i)}{f(\rho_i)\sqrt{2h_*}}
= - \frac{\Lambda'(\rho_i)}{f'(\rho_i)\sqrt{2h_*}}$, whereas the proportionality coefficient in view of~\eqref {eq3.2} equals $h'(\rho_i)/k'(\rho_i) = \frac{k_*-\Lambda(\rho_i)}{f^2(\rho_i)} = 
-\frac{\Lambda'(\rho_i)}{f(\rho_i)f'(\rho_i)}$, so it has the same sign.

(b) Let us prove the relation $U_{k(\rho_i)}''(\rho_i) = -\sigma_i\dfrac{\sqrt{2h}}{f(\rho_i)}g''_{\sigma_i}(\rho_i)$ showing a connection between the type of the critical point of the function $g_{\pm}(r)$, the type of the critical circle $\mathcal O_{r,k(r)}$ and the sign of this circle, see Table~\ref{table1}. Due to the relations $g_{\sigma_i}(r)=\Lambda(r) + \sigma_i \sqrt{2h} f(r)$ and $\sigma_i=-\frac{\Lambda'(r)}{f'(r)\sqrt{2h}}$, we have $g_{\sigma_i}''(\rho_i)=\Lambda''(\rho_i)-\dfrac{\Lambda'(\rho_i)}{f'(\rho_i)}f''(\rho_i)$, and from~\eqref {eq3.6} we obtain the required relation between the second derivatives.

(c) One also easily verifies that the Hesse matrices $d^2(H|_{\{K=k_*\}})$ and $d^2(K|_{\{H=h_*\}})$ at the given critical point are proportional with a proportionality coefficient having the sign $-\sigma_i$:
since $dH=\lambda dK$ at this point, the Hesse matrix $d^2(H-\lambda K)$ is a well-defined quadratic form, therefore its restriction to the tangent plane to the surface $\{K=k_*\}$ (coinciding with the tangent plane to the surface $\{H=h_*\}$, due to the proportionality of $dH$ and $dK$) is equal to $d^2(H|_{\{K=k_*\}})=-\lambda d^2(K|_{\{H=h_*\}})$.
\end{remark}
\begin{proof}[Proof of Theorem~\ref{th4.2}]
Let's prove the first part of the theorem. By Definition~\ref {def2.2} (see~\eqref{eq2.3}) of the magnetic geodesic flow, the energy level $Q^3_h$ does not depend on the magnetic field, i.e., the isoenergy 3-manifold $Q^3_h$ is the same as that of the geodesic flow on the given Riemannian manifold $(M,\,ds^2)$. From the general properties of geodetic flows, we obtain that:
	\begin{itemize}
\item $Q^3_h$ is non-singular if and only if $h\neq 0$,
\item for any $h>0$, the isoenergy 3-manifold $Q^3_h$ is diffeomorphic to the set of unit tangent vectors to $M$, i.e., to the spherical tangent bundle $STM$ of the manifold $M$.
 \end{itemize}
In the case of $M = S^2$, we get $Q^3_h \cong STS^2\cong SO(3) \cong\mathbb{R}P^3$.

Let us show that the Reeb graph of the function $K$ on a non-singular $Q^3_h$ is a tree. Suppose the contrary, then there is a cycle in it. 
This means that there is a normal subgroup in $\pi_1(Q^3_h)$, whose quotient group is isomorphic to $\mathbb{Z}$. But this is impossible, since $\pi_1(\mathbb{R}P^3)\cong \mathbb{Z}_2$.
 
The remaining statements of the theorem follow from Corollary~\ref {cor3.12}, Lemma~\ref {le4.5} and Remark~\ref {rem4.6}~(a).
\end{proof}

\begin{remark} \label {rem4.8}
In the proof of Lemma~\ref{le4.5}, we obtained an answer to the question formulated before it. That answer is as follows. In order to find the Fomenko molecule for the Liouville foliation on a non-singular isoenergy manifold $Q^3_h$ of the system $S(f,\Lambda)$, it is enough to know the value $h$, the functions $f$ and $\Lambda$ and to perform the following constructions (this also follows from~\cite[sect.~4.1]{KO:20}):
    \begin{itemize}
\item construct the functions $g_+(r)$, $g_-(r)$ and the Reeb graph of the function $k$ restricted to the flat region $\Pi_h$ defined in~\eqref{eq4.1};
\item assign 3-atoms to the vertices of this graph in accordance with the types and signs of the critical circles (ordered in ascending order of coordinates $r$), determined from Table~\ref{table1}.
    \end{itemize}
\end{remark}

\subsection{Calculation of the Fomenko--Zieschang isoenergy marks (global 3D topological invariants)}

As noted above, the Liouville foliation of the system consists of regular fibers, which are two-dimensional tori, and singular fibers (whose small neighborhoods are called 3-atoms). Let's fix a (non-singular) energy value $h>0$ different from the abscissa of cusp points of the curve $\gamma_1$, and construct the Fomenko molecule (see Remark~\ref{rem4.8}) of the function $K \big|_{Q^3_h}$. According to Theorem~\ref{th4.2}, there are 3-atoms $A$ at its terminal vertices, saddle 3-atoms $V=V_{\sigma_1\dots\sigma_m}\times S^1$ at its non-terminal vertices, and the edges correspond to smooth 1-parameter families of tori that ``glue'' boundary tori of the atoms. In order to fully restore the topology of the Liouville foliation, one has to know how does this gluing work.

Let us first define orientations on the isoenergy manifold $Q^3_h$, as well as on boundary tori and on critical circles of 3-atoms on it.
A basis $e_1,e_2,e_3\in T_xQ^3_h$, where $x\in Q^3_h$, is called {\em positively oriented in $Q^3_h$} if 
\begin{equation} \label{eq4.2} 
\widetilde\omega \wedge \widetilde\omega (e_1,e_2,e_3,\grad H)>0.
\end{equation}
Let $X\subset Q^3_h$ be a 3-atom, $x\in\partial X$ and $n\in T_xQ^3_h$ be a vector of the external normal to $\partial X$ in $Q^3_h$. A basis $e_1,e_2\in T_x(\partial X)$ is called {\em positively oriented in $\partial X$} if the triple $(e_1,e_2,n)$ is positively oriented in $Q^3_h$, i.e.,
\begin{equation*}
 \widetilde\omega \wedge \widetilde\omega (e_1,e_2,n,\grad H)>0.
\end{equation*}
Let's define the orientation on critical circles of the Bott function $K\big|_{Q^3_h}$ via the flow of $\sgrad H=\widetilde{\omega}^{-1}dH$ on elliptical circles, and the flow of $\sgrad K=\widetilde{\omega}^{-1}dK$ on hyperbolic ones.

In order to describe the gluing of tori mentioned above, we choose an {\em admissible basis} on each boundary torus of the 3-atom, which is a pair of cycles $(\lambda, \mu)$, according to the following rules.
\begin{enumerate}
\item If the 3-atom is elliptic (i.e., diffeomorphic to a full torus), then a cycle $\lambda$ on its boundary torus is selected that can be contracted to a point inside the full torus (a so-called vanishing cycle). In this case, any cycle that complements $\lambda$ to a basis on the boundary torus can be used as the cycle $\mu$. The orientation of $\mu$ is chosen to be consistent with the flow of $\sgrad H$ on the critical circle, and the orientation of $\lambda$ is chosen so that the pair $(\lambda, \mu)$ is positively oriented on the boundary torus.
	
\item If the 3-atom is saddle and its separatrix diagram is orientable (i.e., the 3-atom is a trivial $S^1$-bundle over a surface $P$ with a given Morse function on it), then the cycle $\lambda$ is chosen to be homotopic to a fiber of the Seifert fibration on this atom, and its orientation is given by the flow of $\sgrad K$. Let's fix a cross-section $P$ and choose as $\mu_i$ the intersections of $P$ with the boundary tori of the 3-atom. One can choose $\mu_i$ independently on each torus, but the algebraic sum of all cycles $\mu_i$ must not contain $\lambda$, i.e., $\mu_i$ are related by the condition of the existence of a cross-section of the 3-atom passing through them. The orientation of each cycle $\mu_i$ is chosen so that the pair $(\lambda, \mu_i)$ is positively oriented on the corresponding boundary torus of the 3-atom.	
\end{enumerate}
After choosing an admissible basis on each boundary torus, one obtains a gluing matrix for each pair of tori of the molecule connected by an edge. Using these matrices, one computes marks for this edge.

Let's formulate a theorem about these marks.

\begin{theorem} \label{th4.8} 
If the pair of functions $f,\Lambda$ determining the magnetic geodesic flow $S(f,\Lambda)$ on the sphere satisfies the conditions 1--6 from Assumption~\ref{assump2.4} and the value $h>0$ is different from the abscissa of the cusp points of the curve $\gamma_1$ and of the intersection points of this curve with the lines $\{k=\Lambda(0)\}$ and $\{k=\Lambda(L)\}$, then $K\big|_{Q^3_h}$ is a Bott function, and the following holds:
\begin{enumerate}		
\item the Fomenko molecule of the Bott function $K\big|_{Q^3_h}$ is a tree with 3-atoms $A$ at the terminal vertices and saddle 3-atoms $V=V_{\sigma_1\dots\sigma_k} \times S^1$ at the other vertices;

\item the topological Fomenko--Zieschang invariant of the Liouville foliation on the isoenergy manifold $Q^3_h$ (with the orientation of hyperbolic critical circles given by the flow of $\sgrad K=\widetilde\omega^{-1}dK$) is given by the Fomenko molecule with the following marks:
\begin{itemize}
\item $r=0$, $\varepsilon=\pm1$ on the edges of the form $A-V$, where the sign ``$-$'' is taken in the case when the 3-atom $A$ corresponds to a local minimum of the function $K\big|_{Q^3_h}$, and the sign ``$+$'' is taken in the case when the 3-atom $A$ corresponds to a local maximum of the function $K\big|_{Q^3_h}$;

\item $r=\infty$, $\varepsilon=1$ on the edges of the form $V-V'$;
      
\item if the Fomenko molecule has the form $A-A$, then the marks on the single edge are $r=1/2$, $\varepsilon=1$;

\item all saddle atoms form a single family with the mark $n=2$.
  \end{itemize}
   \end{enumerate}
\end{theorem}

\begin{proof}
The first statement of the theorem has already been proved in Theorem~\ref{th4.2} and it is given as a reminder.

Let's prove the second statement.

Note that the boundaries of all 3-atoms in the isoenergy manifold $Q^3_h$ are divided into positive and negative ones as follows. Two directions are given on any boundary torus: the direction of growth of $K$ and the direction of the external normal to this boundary torus. We will call the boundary torus {\em positive} if these directions coincide, and {\em negative} if they are opposite.
	
Let's fix the orientation on the edges of the molecule given by the direction of growth of the additional first integral $K$. For each edge, we define an admissible  basis $(\lambda^-,\mu^-)$ on the corresponding positive boundary torus of the 3-atom standing at the beginning of the edge, and an admissible basis $(\lambda^+,\mu^+)$ on the corresponding negative torus of the 3-atom standing at the end of the edge.

Let's fix a boundary torus of a 3-atom. It corresponds to values $H=h=\const$, $K=k=\const$. Let's define two cycles on it:
 \begin{equation*}
		\begin{split}
			&\alpha_r=\left\{p_r=\pm\sqrt{2(h-U_k(r))},\,K=k,\,r\in \Delta_i\subseteq R_{h,k},\,\varphi=\const\right\},\\
			&\alpha_\varphi=\left\{p_r=\const,\,K=k,\,r=\const,\,\varphi \in \mathbb{R}/2\pi\mathbb{Z}\right\},
		\end{split}
	\end{equation*}
and express the basic cycles in terms of them. Here the cycle $\alpha_\varphi$ is oriented by the direction of the vector field $\sgrad K=\widetilde\omega^{-1}dK=(0,0,0,1)$, and the cycle $\alpha_r$ is oriented by the direction of growth of the variable $r$ on its arc $\alpha_r\cap\{p_r>0\}$, and by the direction of decrease of $r$ on its arc $\alpha_r\cap\{p_r<0\}$.

Since $\widetilde\omega \wedge \widetilde\omega=-2 d p_r \wedge dK \wedge dr \wedge d\varphi$, we have
\begin{equation}\label {eq4.3} 
		\widetilde\omega \wedge \widetilde\omega (v_{\alpha_r},v_{\alpha_\varphi},\grad K,\grad H)=-2 \det\left[v_{\alpha_r},v_{\alpha_\varphi},\grad K,\grad H\right],
	\end{equation}
where the right-hand side of this formula contains the determinant of the matrix composed of columns of the coordinates of the corresponding vectors; $v_{\alpha_r}$ and $v_{\alpha_\varphi}$ denote the velocity vectors of the curves $\alpha_r$ and $\alpha_\varphi$.
In order to find the sign of the expression~\eqref{eq4.3}, let us explicitly compute the vectors included in it:
	\begin{equation*}
		\begin{split}
		&\grad H=\left(p_r,\dfrac{k-\Lambda(r)}{f^2(r)},\beta,0 \right), \quad \grad K=(0,1,0,0), \\
		&v_{\alpha_r}=\left(-\dfrac{\beta}{p_r},0,1,0\right), \quad v_{\alpha_\varphi}=(0,0,0,1).
		\end{split}
	\end{equation*}
Here $\beta=U_k'(r)$, and the vector $v_{\alpha_r}$ is computed on the arc $\alpha_r\cap\{p_r>0\}$. Hence, the value of~\eqref {eq4.3} is equal to $2(\beta^2+p_r^2)/p_r>0$ (this is consistent with the inequality $\widetilde\omega\wedge\widetilde\omega (\sgrad H,\sgrad K,\grad K,\grad H)>0$).

An admissible basis on the boundary torus of a 3-atom $A$ can be chosen in the form $\lambda_A= -\alpha_r$, $\mu_A=\mp\alpha_\varphi$, and in the form $\lambda_V=\alpha_\varphi$, $\mu_V=\mp\alpha_r$ on the boundary tori of a 3-atom $V$. Here, the cycle on a positive boundary of the 3-atom is taken with an upper sign, the cycle on a negative one with a lower sign. We also remark that the sign $\mp$ in the formula for the cycle $\mu_A$ is 
$\sigma=\sgn(k'/h')$, the sign of the critical circle of the 3-atom (see Remark~\ref {rem4.6} to Table~\ref{table1}).

Each edge of the molecule corresponds to a one-parameter family of Liouville tori parameterized by the value $k$ of the function $K$. If, for some $k=k_D$, the corresponding connected component of the domain of possible motion $R_{h,k}$ contains exactly one pole of the sphere (i.e., it is homeomorphic to the disk, thus $k\in\{\Lambda(0),\ \Lambda(L)\}$ due to Remark \ref {rem:Delta}),
then in a small neighborhood of the value $k_D$ the cycle $\alpha_r^+$ for $k>k_D$ and the cycle $\alpha_r^-$ for $k<k_D$ are related by the equality
	\begin{equation*}
		\alpha_r^{+}=\alpha_r^- -\alpha_\varphi,
	\end{equation*}
which can be proved similarly to~\cite[eq.~(12)]{KZAnt}.\footnote{Let's show this numerically. We will assume that $R_{h,k_D}$ contains the north pole ($r=0$). The cycle $\alpha_r$ is homologous to the difference of two curves: the arc $\gamma_H^{[r_1,r_2]}$ of the trajectory of the vector field $\sgrad H=\widetilde\omega^{-1}dH$, $r\in[r_1,r_2]$, and the arc $\gamma_K^{[0,\Phi]}$ of the cycle $\alpha_\varphi$, $\varphi\in[0,\Phi]$, where $\Delta_i=[r_1,r_2]$ is the projection to the coordinate $r$ of the corresponding connected component of the domain of possible motion $R_{h,k}$, $r_i=r_i(h,k)$, $\Phi=\Phi(r_1,r_2)$ is the increment of the $\varphi$-coordinate when moving along the arc $\gamma_H^{[r_1,r_2]}$. This increment is $\Phi(r_1,r_2)=2\int_{r_1}^{r_2}\frac{\dot\varphi}{\dot r}dr = 2 \int_{r_1}^{r_2}\frac{\partial_{k}U_k(r)}{\sqrt{2(h-U_k(r))}}dr$. 
Denoting $\varepsilon= k-\Lambda(0)$, $\delta=|\varepsilon|^{3/4}$ and replacing the variable $r =\frac{\varepsilon v}{\sqrt{2h}}$, we have $r_1<\delta<r_2$ and $\Phi(r_1,\delta)=2\int_{r_1}^\delta \frac{(k-\Lambda(r))dr}{f(r)\sqrt{2hf^2(r)-(k-\Lambda(r))^2}}=2\sgn \varepsilon \int_{r_1}^\delta \frac{dr}{f(r)\sqrt{2h(\frac{f(r)}{k-\Lambda(r)})^2-1}}
\approx 2\sgn \varepsilon \int_{1}^\infty \frac{dv}{v\sqrt{v^2-1}} = \pi \sgn \varepsilon$. Therefore $\alpha_r^- \approx \gamma_H^{[\delta,r_2(h,k_D)]} - \gamma_K^{[0,\Phi(\delta,r_2(h,k_D))-\pi]}$, 
$\alpha_r^+ \approx \gamma_H^{[\delta,r_2(h,k_D)]} - \gamma_K^{[0,\Phi(\delta,r_2(h,k_D))+\pi]}$. This shows that $\alpha_r^+ = \alpha_r^- - \alpha_\varphi$ due to the continuity of the functions $r_2(h,k)$ and $\Phi(\delta,r_2)$.}
If for some $k=k_S$, the domain of possible motion $R_{h,k}$ contains both poles of the sphere (i.e., it is homeomorphic to the sphere, thus $k=\Lambda(0)=\Lambda(L)$ due to Remark \ref{rem:Delta}), then in a small neighborhood of the value $k_S$ the cycle $\alpha_r^+$ for $k>k_S$ and the cycle $\alpha_r^-$ for $k<k_S$ are related in a different way:
\begin{equation*}
		\alpha_r^{+}=\alpha_r^- -2\alpha_\varphi.
	\end{equation*}
The cycle $\alpha_\varphi$ is well-defined for all $k$.

Taking into account the above, let us compute the gluing matrix in each case.
	\begin{enumerate}
\item Consider an edge $A\to V$.

On the 3-atom $A$, the cycle $\alpha_r$ is vanishing, so an admissible basis can be chosen in the form $\lambda_A=-\alpha_r$, $\mu_A=-\alpha_\varphi$. The 3-atom $V$ has a global cross-section $\varphi=\const$, so $\lambda_V=\alpha_\varphi$, $\mu_V=\alpha_r$.
		
If there are no values of $k_D$ and $k_S$ on this edge, then
		\begin{equation*}
			\begin{pmatrix}
				\lambda_V \\ \mu_V
			\end{pmatrix}=
			\begin{pmatrix}
				0 & -1 \\ -1 & 0
			\end{pmatrix}
			\begin{pmatrix}
				\lambda_A \\ \mu_A
			\end{pmatrix}.
		\end{equation*}
If only one value of $k_D$ is present on the edge, then we can regard $\lambda_A$ and $\mu_V$ as $\lambda_A=-\alpha_r^-$ and $\mu_V=\alpha_r^+=\alpha_r^--\alpha_\varphi = -\lambda_A+\mu_A$, hence the relation between the basic cycles will be 
		\begin{equation*}
			\begin{pmatrix}
				\lambda_V \\ \mu_V
			\end{pmatrix}=
			\begin{pmatrix}
				0 & -1 \\ -1 & 1
			\end{pmatrix}
			\begin{pmatrix}
                \lambda_A \\ \mu_A
			\end{pmatrix}.
		\end{equation*}
And if there are two values $k_D$ or one $k_S$, then $\lambda_A = -\alpha_r^-$, $\mu_V = \alpha_r^+ = \alpha_r^- -2\alpha_\varphi = -\lambda_A + 2\mu_A$, and 
		\begin{equation*}
			\begin{pmatrix}
				\lambda_V \\ \mu_V
			\end{pmatrix}=
			\begin{pmatrix}
				0 & -1 \\ -1 & 2
			\end{pmatrix}
			\begin{pmatrix}
                \lambda_A \\ \mu_A
			\end{pmatrix}.
		\end{equation*}
	
		\item Consider an edge $V\to A$.
		
Similarly to the previous case, we get the gluing matrices $\left(\begin{smallmatrix}
0 & 1 \\ 1 & 0
\end{smallmatrix}\right)$, 
$\left(\begin{smallmatrix}
  1 & 1 \\ 1 & 0
\end{smallmatrix}\right)$ and
$\left(\begin{smallmatrix}
  2 & 1 \\ 1 & 0
\end{smallmatrix}\right)$ respectively.

\item Consider an edge $V\to V'$.

Admissible bases on the boundary tori of the saddle atoms are chosen as follows: $\lambda_V=\alpha_\varphi$, $\mu_V=-\alpha_r$, $\lambda_{V'}=\alpha_\varphi$, $\mu_{V'}=\alpha_r$.
		
If there are no values of $k_D$ and $k_S$ on this edge, then
\begin{equation*}
			\begin{pmatrix}
				\lambda_{V'} \\ \mu_{V'}
			\end{pmatrix}=
			\begin{pmatrix}
				1 & 0 \\ 0 & -1
			\end{pmatrix}
			\begin{pmatrix}
				\lambda_V \\ \mu_V
			\end{pmatrix}.
		\end{equation*}
In the case of only one value $k_D$ on the edge, we have 
$\mu_V=-\alpha_r^-$, $\mu_{V'}=\alpha_r^+ = \alpha_r^- - \alpha_\varphi = -\mu_V - \lambda_V$, thus
		\begin{equation*}
			\begin{pmatrix}
				\lambda_{V'} \\ \mu_{V'}
			\end{pmatrix}=
			\begin{pmatrix}
				1 & 0 \\ -1 & -1
			\end{pmatrix}
			\begin{pmatrix}
				\lambda_V \\ \mu_V
			\end{pmatrix}.
		\end{equation*}
In the case of two values $k_D$ or one $k_S$, we have
$\mu_V=-\alpha_r^-$, $\mu_{V'}=\alpha_r^+ = \alpha_r^- - 2\alpha_\varphi = -\mu_V - 2\lambda_V$, thus
		\begin{equation*}
			\begin{pmatrix}
				\lambda_{V'} \\ \mu_{V'}
			\end{pmatrix}=
			\begin{pmatrix}
				1 & 0 \\ -2 & -1
			\end{pmatrix}
			\begin{pmatrix}
				\lambda_V \\ \mu_V
			\end{pmatrix}.
		\end{equation*}
\item Consider an edge $A\to A$. Denoting this edge by $A\to A'$, we have $\lambda_A= -\alpha_r^-$, $\mu_A=\alpha_\varphi$ and $\lambda_{A'}= -\alpha_r^+ = - \alpha_r^- + 2\alpha_\varphi = \lambda_A + 2\mu_A$, $\mu_{A'}=-\alpha_\varphi = -\mu_A$, thus
		\begin{equation*}
			\begin{pmatrix}
				\lambda_{A'} \\ \mu_{A'}
			\end{pmatrix}=
			\begin{pmatrix}
				1 & 2 \\ 0 & -1
			\end{pmatrix}
			\begin{pmatrix}
				\lambda_A \\ \mu_A
			\end{pmatrix}.
		\end{equation*}
	\end{enumerate}

The first row of the gluing matrix yields that the marks on the edges $A\to V$ are equal to $r=0$, $\varepsilon=-1$, 
the marks on the edges $V\to A$ are equal to $r=0$, $\varepsilon=1$, 
the marks on the edges $V-V'$ are equal to $r=\infty$, $\varepsilon=1$, 
the marks on an edge $A-A$ are equal to $r=1/2$, $\varepsilon=1$,\footnote{In the case of the Fomenko molecule $A-A$, the mark $r$ can be found alternatively as follows. In this case, the manifold $Q^3_h\cong\mathbb{R}P^3$ is glued from two full tori. One can show that the mark $r$ equals $r=1/2$ in this case.}
and all saddle atoms form a single family\footnote{Recall~\cite[sect.~4.3]{IGS} that if the first row of the gluing matrix corresponding to an edge is equal to $(a,b)$, then the marks on this edge are defined by the formulas $r=\frac ab\mod 1\in(\mathbb Q/\mathbb Z)\cup\{\infty\}$, $\varepsilon=\sgn b$ in the case of $b\ne0$, $\varepsilon=\sgn a$ in the case of $b=0$.
If the molecule is cut along all edges with the marks $r\ne\infty$, then the molecule will break up into several connected components. The connected components that do not contain atoms of type $A$ are called {\em families}.}. Let's compute the mark $n$ on this family. Note that the presence of one value $k_D$ on an edge means that the contribution of this edge to the mark $n$ is $1$, and the presence of two values $k_D$ or one $k_S$ means the contribution of $2$~\cite[sect.~4.3]{IGS}. Therefore $n=2$ in all cases.
	
The theorem is proved.
\end{proof}

\begin{remark}
Our computation of the mark $n=2$ was done for the orientation of $Q^3_h$ given by the formula~\eqref{eq4.2}. It is not difficult to show (see~\cite[vol.~1, sect.~4.5.2]{IGS}) that when the orientation of $Q^3_h$ is changed to the opposite, the mark $n$ will change the sign.
\end{remark}

\section{Description of all bifurcation diagrams for magnetic geodesic flows $S(f,\Lambda)$. Global 4D topological invariants} \label {s5}

Recall the definitions of some singular points of planar curves.

\begin{definition} \label{def5.1}
A point of a planar curve $\gamma$ at which its velocity vector does not vanish (i.e., $\dot\gamma(t)\ne 0$) and the velocity and acceleration vectors are collinear is called a {\em point of straightening}. A point of straightening of a planar curve $\gamma$ is called an {\em inflection point} if the first and third derivatives at this point are linearly independent.
\end{definition}

\begin{definition} \label{def5.2} 
A point of a planar curve $\gamma$ at which its velocity vector vanishes (i.e., the first derivative $\dot\gamma(t) = 0$) is called a {\em singular point}. A singular point of a planar curve $\gamma$ is called an {\em (ordinary) cusp point} if the second and third derivatives at this point are linearly independent.
\end{definition}

\subsection{Projectively dual projective curves.
Singular points and inflection points of good projective curves} \label{su5.1}

Let $\gamma(t)=(x(t),y(t),z(t))\in\mathbb{R}^3$ be a $C^\infty$-smoothly parametrized curve in $\mathbb R^3$. Let's define auxiliary functions, namely: the parametrized curve
$$
\gamma^*(t)=\left[\gamma(t), \dot{\gamma}(t)\right] \in \mathbb R^3
$$
and the scalar function
$$
\varkappa_\gamma(t)=\langle\gamma^*(t),\ddot{\gamma}(t)\rangle=\left(\gamma(t),\dot{\gamma}(t),\ddot{\gamma}(t)\right) \in \mathbb R.
$$
Here $\langle,\rangle$ and $[,]$ denote the scalar product and vector product, respectively, of vectors in $\mathbb R^3$.

Suppose that $\gamma(t)$ does not pass through the origin.
Let $P\gamma(t)=(x(t):y(t):z(t))\in\mathbb R P^2$ be the projectivization of the curve $\gamma(t)$, and $P(\gamma^*)(t)$ be the projectivization of the curve $\gamma^*(t)$. The projective curve $P(\gamma^*)$ is called {\em projectively dual} to the projective curve $P\gamma$.

We will show below  (see~\eqref{eq5.1} and Lemma~\ref{le5.8}) that projective duality is well-defined in the case of so-called ``good'' curves (see Definition~\ref{def5.5}).\footnote{Projective duality extends to submanifolds of a projective space of any dimension and is often used for algebraic projective varieties (see, for example, \cite{Tev:2000}). However, the functions and curves considered in this paper are not assumed to be algebraic or real-analytic, but only $C^\infty$-smooth. The projective duality of curves naturally extends to a wide class of $C^\infty$-smooth curves that are not necessarily real-analytic. In this paper, projective duality is considered for ``good'' curves. Such curves are ``typical'' in the sense that they form an open and dense subspace with respect to the $C^\infty$-topology in the space of all $C^\infty$-smooth projective curves.}
From Definitions~\ref{def5.1} and~\ref{def5.2}, one obtains the following definitions of similar concepts for projective curves. 

\begin{definition}\label{def5.3}
A point of the projective curve $P\gamma$ is called a {\em point of straightening} if $\gamma^*(t)\ne 0$ and $\varkappa_\gamma(t) = 0$ at this point. A point of straightening of the projective curve $P\gamma$ is called an {\em inflection point} if $\dot\varkappa_\gamma(t)\ne 0$ at this point.
\end{definition}

\begin{definition} \label{def5.4}
A point of the projective curve $P\gamma$ is called a {\em singular point} if $\gamma^*(t)= 0$ at this point. A singular point of the projective curve $P\gamma$ is called a {\em cusp point} (or a {\em return point}) if $\ddot\varkappa_\gamma(t)\ne 0$ at this point.
\end{definition}

\begin{definition} \label{def5.5}
A $C^\infty$-smoothly parametrized curve $\gamma(t)$ in $\mathbb R^3\setminus\{0\}$ and its projectivization $P\gamma(t)$ are called {\em good} if all points of straightening of the curve $P\gamma(t)$ are inflection points, and all its singular points are cusp points.	
\end{definition}

\begin{remark} \label{rem5.6}
The zeros of the function $\varkappa_\gamma(t)$ are either singular points (with $\gamma^*(t)=0$ at them) or points of straightening (with $\gamma^*(t)\ne 0$ and $\varkappa_\gamma(t) = 0$ at them). This implies that $\varkappa_\gamma(t)$ has simple zeros at the inflection points (and only at them), while $\varkappa_\gamma(t)$ has a zero of multiplicity 2 and $\gamma^*(t)=0$ at the cusp points (and only at them).
These properties are true for all curves, not only for good ones.
\end{remark}

\begin{remark} \label{rem5.7}
The notions of a point of straightening, an inflection point, a singular point and a cusp point of the projective curve $P\gamma$ (and therefore the definition of a good projective curve $P\gamma$) are well-defined, i.e., they do not depend on the choice of an affine reprezentative $\gamma$ of this projective curve. Indeed, for any smooth function $\varphi=\varphi(t)$ having a constant sign and no zeros, the curves $\gamma$ and $\varphi\gamma$ are both good or not good simultaneously, furthermore the points of straightening and singular points (together with their types) for the corresponding (coinciding) projective curves $P\gamma$ and $P(\varphi\gamma)$ coincide, due to Remark~\ref{rem5.6} and the relations $(\varphi\gamma)^* = \varphi^2 \gamma^*$ and $\varkappa_{\varphi\gamma} = \varphi^3 \varkappa_\gamma$. The latter relations can be proved as follows:
	\begin{equation} \label{eq5.1}
		\begin{array}{ll} 
		& (\varphi\gamma)^* = [\varphi\gamma, \dot\varphi\gamma + \varphi\dot\gamma] = \varphi^2 [\gamma, \dot\gamma] = \varphi^2 \gamma^*, \\
		& \varkappa_{\varphi\gamma} = \left<\left[\varphi\gamma, \dot\varphi\gamma + \varphi\dot\gamma\right], \ddot\varphi\gamma + 2\dot\varphi\dot\gamma + \varphi\ddot\gamma\right> = \varphi^3 \left<\left[\gamma, \dot\gamma\right], \ddot\gamma\right> = \varphi^3 \varkappa_\gamma.
		\end{array} 
	\end{equation}
\end{remark}	

\begin{lemma}[on projectively dual curves] \label{le5.8}
Let $\gamma(t)$ be a good curve not passing through the origin in $\mathbb R^3$. Then the projectivization of the curve $\gamma^*(t)=\left[ \gamma(t), \dot\gamma(t)\right]$ extends by continuity (to the points corresponding to the singular points of $\gamma^*$) to a good projective curve denoted by $(P\gamma)^*$ and called the {\em projectively dual} of the projective curve $P\gamma$.

The operation $\gamma\mapsto\gamma^*$ on the set of $C^\infty$-smoothly parametrized curves in $\mathbb R^3$ has the following properties:
\begin{enumerate}
    \item for any smooth function $\varphi=\varphi(t) > 0$, one has $(\varphi\gamma)^* = \varphi^2\gamma^*$ and $\varkappa_{\varphi\gamma} = \varphi^3 \varkappa_\gamma$;

    \item $\gamma^{**} = \varkappa_\gamma \gamma$;

    \item $\varkappa_{\gamma^*}=\varkappa_\gamma^2$;

    \item the induced operation $P\gamma\mapsto(P\gamma)^*$ on the set of good projective curves is well-defined and inverse to itself, i.e.,
    $(P\gamma)^{**} = P\gamma$, further $(P(\gamma\circ\tau))^* = (P\gamma)^*\circ\tau$ for any regular reparametrization $t\to\tau(t)$;

    \item for any cusp point of the curve $P\gamma$, the corresponding point of the curve $(P\gamma)^*$ is an inflection point, and vice versa: for each inflection point of the curve $P\gamma$, the corresponding point of the curve $(P\gamma)^*$ is a cusp point.
\end{enumerate}
\end{lemma}
\begin{proof}
The first property has already been proved in~\eqref{eq5.1}.

Let us prove the second property. Since $\gamma^* = \left[\gamma,\dot\gamma\right]$ and $\frac d{dt}(\gamma^*) = \left[\gamma,\ddot\gamma\right]$, we have
\begin{equation*}
\gamma^{**} = \left[\gamma^*,\frac d{dt}(\gamma^*)\right] = \left[ \left[ \gamma, \dot\gamma\right], \left[ \gamma,\ddot\gamma\right] \right].
\end{equation*}
Let us show that the latter expression equals $\varkappa_\gamma \gamma$. This is automatically true at any singular point (since $\left[\gamma, \dot\gamma\right]=0$ and $\varkappa_\gamma=0$ there) and at any point of straightening (since $\varkappa_\gamma=0$ and $\ddot\gamma$ is a linear combination of $\gamma$, $\dot\gamma$).
At any other point, we have $\gamma^*\ne0$ and $\varkappa_\gamma\ne0$, and since the vectors $\left[\gamma, \dot\gamma\right]$ and $\left[\gamma,\ddot\gamma\right]$ are perpendicular to the nonzero vector $\gamma$, their vector product is proportional to the vector $\gamma$. To find the proportionality coefficient, we consider the orthonormal basis $e_1,e_2,e_3$ obtained from the basis $\gamma,\dot\gamma,\ddot\gamma$ (at a fixed point of the curve) by orthogonalization. Then at this point we have (for some real constants $a$, $b_i$, $c_j$)
\begin{equation*}
\gamma=ae_1,\quad
\dot\gamma=b_1e_1+b_2e_2,\quad
\ddot\gamma=c_1e_1+c_2e_2+c_3e_3,
\end{equation*}
hence
$\left[\gamma, \dot\gamma\right]=ab_2e_3$, $\varkappa_\gamma=ab_2c_3$,
$\left[\gamma, \ddot\gamma\right]=ac_2e_3-ac_3e_2$.
Therefore, the vector product
$\left[\left[\gamma, \dot\gamma\right], \left[\gamma, \ddot\gamma\right]\right] = a^2 b_2 c_3 e_1 = a b_2 c_3 \gamma = \varkappa_\gamma \gamma$, as required.

The third property follows from the second:
\begin{equation*}
\varkappa_{\gamma^*} 
= \left<\gamma^{**}, \frac{d^2}{dt^2}(\gamma^*)\right> 
= \left<\varkappa_\gamma \gamma, \left[\dot\gamma,\ddot\gamma\right] + \left[\gamma,\frac{d^3}{dt^3}\gamma\right]\right> = \varkappa_\gamma^2.
\end{equation*}

Let us prove the fifth property.

Consider an inflection point of the curve $P\gamma$. By the second property, we have
$\left[\gamma^*,\frac{d}{dt}(\gamma^*)\right] = \gamma^{**} = \varkappa_\gamma \gamma = 0$ at this point, thus this point is a singular point of the curve $P(\gamma^*)$. Further, since the function $\varkappa_\gamma$ has a simple zero at this point, by the third property the function $\varkappa_{\gamma^*} = \varkappa_\gamma^2$ has a zero of multiplicity 2 at it. And this is exactly the definition of a cusp point of the curve $P(\gamma^*)$.

Let us now consider a cusp point of the curve $P\gamma$. We have $\gamma^*=\left[\gamma,\dot\gamma\right]=0$ at it, therefore the projectivization of the curve $\gamma^*$ is, generally speaking, not defined at this point.
Further, the function $\varkappa_\gamma$ has a zero of multiplicity 2 at this point. 
By the Hadamard lemma, there exists a smooth function $\mu$ near this point such that $\varkappa_\gamma=\pm \mu^2$. It is clear that the function $\mu$ has a simple zero at this point, and therefore by the Hadamard lemma the curve $\gamma^*/\mu$ extends by continuity to a smooth curve near this point. Let us show that this curve does not pass through the origin and its projectivization has an inflection point at this point. To do this, let us verify that the function $\varkappa_{\gamma^*/\mu}$ has a simple zero at this point. According to the properties 1 and 3 and the construction of the function $\mu$, in a small punctured neighborhood of this point we have
\begin{equation*}
\varkappa_{\gamma^*/\mu} = \dfrac{1}{\mu^3}\varkappa_{\gamma^*}
= \dfrac{1}{\mu^3}\varkappa_\gamma^2 = \dfrac{\mu^4}{\mu^3} = \mu.
\end{equation*}
We conclude, taking into account that the function $\mu$ has a simple zero at this point (see above), that the curve $\gamma^*/\mu$ has a non-zero limit at this point (therefore, the projectivization of the curve $\gamma^*/\mu$ has a limit at this point, thus $(P\gamma)^*$ is well-defined and smoothly parametrized near this point) and by Remark~\ref{rem5.6} the curve $(P\gamma)^*$ has an inflection point at it. The fifth property is proved.

For proving the fourth property, it remains to show that if the curve $\gamma$ is good, then after continuous extension of the projective curve $P(\gamma^*)$ to the points corresponding to the cusp points of the curve $P\gamma$, we obtain a good projective curve. Near the points corresponding to cusp points and inflection points of the curve $P\gamma$, this follows from the fifth property. On the complement of these points, the function $\varkappa_\gamma$ has no zeros, and therefore (in view of the third property) $\varkappa_{\gamma^*} = \varkappa_\gamma^2$ also has no zeros. Therefore, the curve $P(\gamma^*)$ has neither singular points nor straightening points outside the indicated points. This means that the continuous extention of the projective curve $P(\gamma^*)$ under consideration is good.
The equality $(P\gamma)^{**} = P\gamma$ follows from the second property. The equality $(\gamma\circ\tau)^* = \dot\tau \gamma^*\circ\tau$ implies $(P(\gamma\circ\tau))^* = (P\gamma)^*\circ\tau$.

Lemma~\ref{le5.8} is completely proved.
\end{proof}

Let us look at a good projective curve in various affine charts.
Recall that $\gamma(t)=(x(t),y(t),z(t))$.

\begin{remark}\label{rem5.9}
Consider the affine chart $\mathbb{R}P^2\setminus\{z=0\} \cong \mathbb{R}^2$ with the affine coordinates $(x,y)$ (i.e., we set $z=1$). The affine representation of the projective curve $P\gamma \setminus \left\{z=0\right\}$ in this affine chart is a planar curve with the parameterization $(x(t)/z(t), y(t)/z(t), 1)=\varphi(t)\gamma(t)$, where $\varphi(t)=1/z(t)$. 
The oriented curvature of the planar curve $\varphi\gamma$ at any regular point of the curve $P\gamma \setminus \{z=0\}$ in view of~\eqref{eq5.1} equals
\begin{equation} \label{eq5.2}
\dfrac{\varkappa_{\varphi\gamma}}{|\frac{d}{dt}(\varphi\gamma)|^3} = \dfrac{\varphi^3 \varkappa_\gamma}{|\frac{d}{dt}(\varphi\gamma)|^3}.
\end{equation}
Therefore this planar curve (as well as any other affine representation of the projective curve $P\gamma$) on the complement of the set of zeros of the function $\varkappa_\gamma$ is regular and its curvature does not vanish.
It follows from Remark~\ref{rem5.6} and formula~\eqref{eq5.2} 
that the planar curve $\varphi\gamma$ is regular near inflection points and its oriented curvature has simple zeros at them, while the velocity vector has first-order zeros at cusp points and the function $\varkappa_\gamma$ has second-order zeros at them. Consequently, the sign of the oriented curvature changes when passing through an inflection point, but preserves when passing through a cusp point.
\end{remark}

\begin{remark}\label{rem5.10}
At a cusp point of the curve $P\gamma$, the velocity vector of the planar curve $\varphi\gamma$ vanishes and the vectors of the second and third derivatives are linearly independent. This means that the planar curve $\varphi\gamma$ has the following behavior near a cusp point $P\gamma(t_0)$:
\begin{equation*}
\varphi(t)\gamma(t) = \varphi(t_0)\gamma(t_0) + (t-t_0)^2 e_1 + (t-t_0)^3 e_2 + O((t-t_0)^4),\quad t \to t_0,
\end{equation*}
where $e_1,e_2$ is a basis of the tangent space to the plane $z=1$ at the point $\varphi(t_0)\gamma(t_0)$. This means that the planar curve $\varphi\gamma$ has a singularity of the ``semi-cubic parabola'' type at this point.
It follows that when passing through a cusp, the velocity vector of the projective curve $P\gamma$ switches its direction to the opposite. Therefore, for any good curve $\gamma(t)$, the tangent line to the projective curve $P\gamma$ is well defined at any point of $P\gamma(t)$, including cusp points. 
We also obtain that the tangent line depends smoothly on the parameter $t$. By Remark~\ref {rem5.9}, the oriented curvature preserves sign when passing through a cusp point.
\end{remark}

\subsection{Geometric characterization of the planar curve $\Gamma=(f,\Lambda,1)$ determining the system $S(f,\Lambda)$}

Let $f$, $\Lambda$ be smooth functions on an interval $\left[0,L\right]$. Let us recall three sets of points from Notation \ref {not3.2}. Let $r_1<\dots<r_n$ be the critical points of $f(r)$ and $r_1^*=0 < r_2^*<\dots<r_N^*=L$ be the critical points of $\Lambda(r)$, put $I=[0,L]\setminus\{r_1,\dots,r_n\}$. Let $r_1^\circ<\dots<r_m^\circ$ be the zeros of $f'(r)\Lambda''(r)-f''(r)\Lambda'(r)$. These three sets are finite and pairwise disjoint if conditions 4--6 from Assumption~\ref{assump2.4} are satisfied.

\begin{lemma}[Geometric characterization of the curve $\Gamma=(f,\Lambda,1)$] \label{le5.11}
The fulfillment of conditions 1--6 from Assumption~\ref {assump2.4}
on the functions $f$, $\Lambda$ is equivalent to the fulfillment of the following conditions on the parametrized planar curve $\Gamma=(f,\Lambda,1)$:
\begin{enumerate}
\item[\rm(i)] the planar curve $\Gamma(r) = (f(r), \Lambda(r), 1)$, $0\le r\le L$, is good (in the sense of Definition~\ref{def5.5}) and regularly parametrized;

\item[\rm(ii)] the curve $\Gamma(r)$ is bounded, it is contained in the half-plane $f\geqslant 0$ and intersects the line $f=0$ only at its endpoints $P_0=\Gamma(0)$ and $P_1=\Gamma(L)$;

\item[\rm(iii)] the tangent line to the curve $\Gamma(r)$, $0\leqslant r\leqslant L$, at any of its inflection points is neither vertical nor horizontal;

\item[\rm(iv)] the curve $\Gamma(r)$, $0\leqslant r\leqslant L$, and its reflection $\Gamma_-(r) = (-f(-r), \Lambda(-r), 1)$, $-L\leqslant r\leqslant 0$, with respect to the line $f=0$ are smoothly glued at their endpoints $P_0$ and $P_1$ to a $C^\infty$-smoothly parametrized closed planar curve having a unit speed $|\Gamma'(0)|=|\Gamma'(L)|=1$ at these points.
\end{enumerate}
\end{lemma}

\begin{proof}
Suppose that conditions 1--6 from Assumption~\ref {assump2.4} are fulfilled. Condition 5 means that the curve $\Gamma(r) = (f(r), \Lambda(r), 1)$, $0\le r\le L$, is regular. Condition 6 means that all its straightening points are inflection points. Therefore, the curve is good, and condition (i) is satisfied. Condition (ii) follows from the compactness of the curve $\Gamma$ and from condition 1. 
Conditions 4 (the Morse property of $f$ and $\Lambda$) and 5 (regularity) imply that the curvature of the curve $\Gamma$ does not vanish at the points with vertical or horizontal tangent lines, thus by Remark~\ref{rem5.9} these points are not inflection points, hence condition (iii) holds.
Condition (iv) follows from conditions 1--3.

Let us show the inverse, i.e., that conditions (i)--(iv) imply conditions 1--6. Conditions 1 and 2 follow from (ii) and (iv). Condition 3 ($f'(0)=1$ and $f'(L)=-1$) follows from condition (iv) and the non-negativity of $f(r)$ (see condition (ii)). Condition 4 (the Morse property of $f$ and $\Lambda$) follows from the regularity and goodness conditions (see condition (i)) and condition (iii). Condition 5 (regularity) follows from (i). Condition 6 follows from condition (i) in view of Remark~\ref{rem5.6}.

Lemma~\ref{le5.11} is proved.
\end{proof}

\subsection{Geometric and dynamical properties of the bifurcation curve $\gamma_1=(h,k)$}
Recall that the bifurcation diagram of the magnetic geodesic flow $S(f,\Lambda)$ consists of two curves:
\begin{equation*}
\begin{split}
&\gamma_1(r) = (h(r), k(r)) = \left(\dfrac{1}{2}\left(\dfrac{\Lambda'(r)}{f'(r)}\right)^2,\, \Lambda(r) - f(r)\dfrac{\Lambda'(r)}{f'(r)}\right),\quad r \in I, \\
&\gamma_2(r) = (0, \Lambda(r)),\quad r\in [0,L].
\end{split}
\end{equation*}
Based on the curve $\gamma_1=(h,k)$, we construct a new curve $\gamma=(a,-1,k)=(\pm\sqrt{2h},-1,k)$ obtained from the curve $\gamma_1=(h,k)$ via the (two-valued) mapping
\begin{equation} \label {eq5.3}
R \colon (h,k)\to(\pm\sqrt{2h}, -1, k) = (a, -1, k)
\end{equation}
by choosing the sign in front of the radical so that the condition $k'(r) a'(r) \leqslant 0$ is satisfied (by this condition, the function $a(r)=\pm\frac{\Lambda'(r)}{f'(r)}$ is uniquely determined for $r\in I$ and has the form $a(r) = \frac{\Lambda'(r)}{f'(r)}$, as we will show in  Lemma~\ref{le5.12} below).

Next, we will formulate the properties of the bifurcation diagram in terms of the curve $\gamma$ rather than $\gamma_1$.

\begin{lemma}[Tangent line to $\Gamma$; projective duality of curves $P\Gamma=(f:\Lambda:1)$ and $P\gamma=(a:-1:k)$]
	\label{le5.12}
	The tangent line to the curve $\Gamma = (f, \Lambda, 1)$ at any point $\Gamma(r)$, $r\in I$, is given by the equation
\begin{equation}
\label{eq5.4}
a(r) f - \Lambda + k(r) = 0.
\end{equation}
In particular, the coefficients of this equation are computed by the formulas 
\begin{equation} \label{eq5.5}
a(r)=\dfrac{\Lambda'(r)}{f'(r)},\quad k(r)=\Lambda(r)-f(r)\dfrac{\Lambda'(r)}{f'(r)},\quad r\in I,
\end{equation}
expressing the curve $\gamma(r) = (a(r), -1, k(r))$, $r\in I$, in terms of the original curve $\Gamma(r)$. That is, the equality $P\gamma = (P\Gamma|_I)^*$ of projective curves is true.
\end{lemma}

\begin{proof}
	By Lemma~\ref{le5.8}, the curve projectively dual to $P\Gamma$ is the projectivization of the curve
$\Gamma^*(r)$ = $\left[\Gamma,\Gamma'\right]$ = $(-\Lambda', f', f\Lambda' - f'\Lambda)$,
i.e., the projective curve
\begin{equation}
\label{eq5.6}
(P\Gamma)^* 
= P(\Gamma^*)
= (-\Lambda' : f' : f\Lambda' - f'\Lambda) = \left(\dfrac{\Lambda'}{f'} : -1 : \Lambda - f\dfrac{\Lambda'}{f'}\right) 
= (\hat a: -1: k) 
= P\hat\gamma,
\end{equation}
where $\hat\gamma := (\hat a, -1, k)$, $\hat a := \Lambda'/f' = \pm a$, and the sign is actually +, as will be shown below.

The definition of the projectively dual curve implies its geometric interpretation: the tangent line to a projective curve $P\Gamma$ at a regular point $P\Gamma(r)$ is given by the equation
\begin{equation} \label{eq5.7}
A(r) x + B(r) y + C(r) z = 0,
\end{equation}
where $\Gamma^*(r) = (A(r),B(r),C(r))$ (this follows from the orthogonality of the vector $\Gamma^*(r)$ to the vectors $\Gamma(r)$ and $\Gamma'(r)$). 
We obtain that the tangent line to the planar curve $\Gamma=(f,\Lambda,1)$ at the point $\Gamma(r)$ is given by the linear equation corresponding to the point $(P\Gamma)^*(r) = P\hat\gamma(r)$, i.e., by the equation
$\hat a(r) f - \Lambda + k(r) = 0$ if $f'(r)\ne0$, i.e., if $r\in I$.

By this geometric interpretation of the projectively dual curve (see \eqref{eq5.7}), the tangent line to the planar curve $\hat\gamma=(\hat a,-1,k)$ at the point $\hat\gamma(r)$ is given by the linear equation corresponding to the point $P(\hat\gamma^*)(r) = (P\hat\gamma)^*(r) = (P\Gamma)^{**}(r) = P\Gamma(r)$ in view of  Lemma~\ref{le5.8}, i.e., by the equation $f(r) a - \Lambda(r) + k = 0$. In particular, the slope of the tangent line to the planar curve $\hat\gamma$ is equal to $k'(r) / \hat a'(r) = -f(r) < 0$. On the other hand, $k'(r) / a'(r)$ is also negative by the construction \eqref {eq5.3} of $a(r)$ for all $r\in[0,L]$ except for the finite set of points at which the curve $\hat\gamma$ is singular or the tangent line $f(r) a - \Lambda(r) + k = 0$ to it (see above) is vertical or horizontal, i.e., for all $r\in (0,L)\setminus\{r_i,r_\ell^\circ\}$ in view of Lemma~\ref{le5.8}.
Hence, the sign $\pm$ in the relation $\hat a = \pm a$ equals ``$+$'' for all $r\in(0,L)\setminus\{r_i,r_\ell^\circ\}$.
Hence, $\hat a = a$ and $\gamma=(a,-1,k) = (\hat a,-1,k) = \hat\gamma$. 

Lemma~\ref{le5.12} is proved.
\end{proof}

Thus, by Lemma~\ref {le5.12}, the curve $P\gamma=(a:-1:k)$ is projectively dual to the curve $P\Gamma=(f:\Lambda:1)$. Hence, by Lemma~\ref {le5.8}, the inverse is also true: the curve $P\Gamma=(f:\Lambda:1)$ is projectively dual to the curve $P\gamma=(a:-1:k)$, i.e., the following statement is true.

\begin{lemma}[Tangent line to $P\gamma$] \label{le5.13}
The tangent line to the projective curve $P\gamma$ at any point $P\gamma(r)$, $r\in\left[0,L\right]$ (including the points $P\gamma(r_i):=(1:0:-f(r_i))$ ``at infinity'' and the cusp points $P\gamma(r^\circ_\ell)$) is given by the equation
$$
f(r) a - \Lambda(r) + k = 0.
$$
In particular, the coefficients of this equation have the form
\begin{equation} \label{eq5.8}
f(r) = - \dfrac{k'(r)}{a'(r)}, \quad \Lambda(r) = k(r) - a(r) \dfrac{k'(r)}{a'(r)}, \qquad
r\in I\setminus\{r^\circ_\ell\}.
\end{equation}
This gives an expression of the curve $\Gamma(r) = (f(r), \Lambda(r), 1)$ in terms of the curve $\gamma=(a, -1, k)$ everywhere except for the points $r_i$, $r^\circ_\ell$ (the latter points correspond to the cusp points of the curve $\gamma(r)$ or, equivalently, the inflection points of the curve $\Gamma(r)$, due to Lemma~\ref{le5.8}). Thus, $P\Gamma = (f:\Lambda:1)=(P\gamma)^*=(a:-1:k)^*$.
\end{lemma}

\begin{proof}
	That the tangent line is given by the indicated equation is proved in the proof of Lemma~\ref{le5.12}. It is clear from this equation that the slope 
    of this tangent line equals $k'(r)/a'(r) = -f(r)$ for any $r\in I\setminus\{r_\ell^\circ\}$. From the equation of the tangent line and the fact that it contains the point $P\gamma(r)=(a(r):-1:k(r))$, we obtain $\Lambda(r)=k(r)+a(r)f(r)$. Substituting the proved formula for $f(r)$ in terms of $a(r)$ and $k(r)$, we obtain the required formula for $\Lambda(r)$ in terms of $a(r)$ and $k(r)$.

Lemma~\ref{le5.13} is proved.
\end{proof}

Let us show (in the following theorem) how, given the bifurcation diagram of the system $S(f,\Lambda)$, we can uniquely construct the bifurcation complex of this system (see the beginning of Sect.~\ref {s3} for its definition) and also determine the types of singular orbits (of ranks $0$ and $1$) corresponding to the vertices and arcs of the bifurcation diagram (i.e., the vertices and edges of the bifurcation complex).

\begin{theorem}[on the relation of geometric and dynamical properties of the bifurcation diagram] \label{th5.13}
\noindent
\begin{enumerate}
\item The magnetic geodesic flow $S(f,\Lambda)$ on the sphere has exactly two points of rank 0, these are the points $(0,N)$ and $(0,S)$, both of which are non-degenerate and have the center-center type.

\item The points of rank 1 form the following two one-parameter families of critical orbits: $\mathcal{C}_1^k=\bigcup\limits_{r\in I \cap (0,L)}
\mathcal{O}_{r,k(r)}$ and $\mathcal{C}_1^\Lambda=\bigcup\limits_{r\in (0,L)}\mathcal{O}_{r,\Lambda(r)}$. The critical orbit $\mathcal{O}_{r,k(r)}$, $r\in I\cap(0,L)$, is non-degenerate if and only if $k'(r) \Lambda'(r) \ne 0$. The latter condition is equivalent to $r\in I\setminus\{r_j^*,r_\ell^\circ\}=[0,L]\setminus\{r_i,r_j^*,r_\ell^\circ\}$. If this orbit is non-degenerate then it has elliptic (respectively hyperbolic) type if and only if the sign of $-k'(r) \Lambda'(r)$ is equal to ``$+$'' (respectively ``$-$''). This sign coincides with the sign of the Gaussian curvature at the point $(f(r), \Lambda(r))$ of the surface obtained by rotating the planar curve $(f,\Lambda)$ about the line $f=0$. The Gaussian curvature at the point $(a(r), k(r))$ of the surface obtained by rotating the planar curve $(a, k)$ about the line $a=0$ has the same sign, where $\gamma_1=(h,k)=(a^2/2,k)$.
The corresponding signs are shown in Fig.~\ref {fig_7}--\ref {fig_9}.

The critical orbit $\mathcal{O}_{r,\Lambda(r)}$, $r\in(0,L)$, is non-degenerate if and only if $\Lambda'(r) \ne 0$. If this orbit is non-degenerate then it is of elliptic type.

\item Degenerate critical orbits have the form $\mathcal{O}_{r,k(r)}$, $r\in\{r_j^*,r_\ell^\circ\}$, i.e., they correspond to the points $\Gamma(r_j^*)$ of the curve $\Gamma=(f,\Lambda,1)$ with horizontal tangent line ($\Lambda'(r_j^*)=0$) and the inflection points $\Gamma(r_\ell^\circ)$ of this curve. The curve $\Gamma$ and the bifurcation curve $\gamma_1$ of the magnetic geodesic flow $S(f,\Lambda)$ have the form shown in Fig.~\ref {fig_7}--\ref {fig_9} near the points corresponding to degenerate singularities, and the form shown in Fig.~\ref {fig_3} near the points corresponding to hyperbolic saddle 3-atoms.

\item The Reeb graph of the function $K|_{\{H=0\}}$ coincides with the Reeb graph of the function $\Lambda(r)$ on $[0,L]$, so it is a broken line $r_1^*,\dots,r_N^*$ obtained from the segment $\left[0,L\right]$ by adding vertices at the critical points $r_j^*$ of the function $\Lambda(r)$, and 
$$
K(\bigcup\limits_{r\in[r_j^*,r_{j+1}^*]} \mathcal{O}_{r,\Lambda(r)}) = \Lambda([r_j^*,r_{j+1}^*]) = \left[\Lambda(r_j^*),\Lambda(r_{j+1}^*)\right].
$$
For any point $p$ of the bifurcation complex, the intersection of a small circular neighborhood of this point in the bifurcation complex with a set of the form $\{h>h(p),\ k=k(p)\}$ is connected.
In other words, for any vertex of the bifurcation complex (respectively any edge corresponding to a 1-parameter family of non-degenerate singular fibers), there exists a unique face of the bifurcation complex adjacent to this vertex (respectively edge) on the right.
\end{enumerate}
\end{theorem}

\begin{proof}
1) The description, non-degeneracy and ellipticity of rank-$0$ points follow from \cite[Prop.~1]{KO:20} or item~1 of Theorem~\ref{th3.3}.

2) The description of rank-$1$ orbits follows from item~2 of Theorem~\ref{th3.3} (which in turn follows from \cite[Prop.~2~(A)]{KO:20}).
The criteria for non-degeneracy of an orbit $\mathcal{O}_{r,k(r)}$ from $\mathcal{C}_1^k$ and the assertion that the type of a non-degenerate critical orbit $\mathcal{O}_{r,k(r)}$ is determined by the sign of $-k'(r)\Lambda'(r)$ follow from item~1 of Theorem~\ref{th3.7} (some of these assertions follow
from \cite[Prop.~2~(A) and~(B)(c)]{KO:20}, see Remark~\ref {rem3.11}).
The coincidence of the signs of the quantity $-k'(r)\Lambda'(r)$ and of the Gaussian curvature of the surface of revolution obtained from the curve $(f,\Lambda)$ follows from the explicit formula for the Gaussian curvature of the surface of revolution (see footnote$^{\ref {footnote1}}$).
Let us show this geometrically: due to~\eqref{eq5.4}, the tangent line to the curve $(f,\Lambda)$ at the point $(f(r), \Lambda(r))$ intersects the line $f=0$ at the point with ordinate $k(r)$, and from the condition $-k'(r)\Lambda'(r)>0$ we obtain that $k(r)$ decreases when moving along the curve in the direction of growth of $\Lambda(r)$, therefore a small neighborhood of the point in the surface of revolution lies on one side of the tangent plane to this surface, which means that the Gaussian curvature is non-negative (the fact that the Gaussian curvature does not vanish is verified separately).
The coincidence of the signs of the quantity $-k'(r)\Lambda'(r)$ and of the Gaussian curvature of the surface of revolution obtained from the curve $(a,k)$ is similarly proved, taking into account that the tangent line to the curve $(a,k)$ at the point $(a(r), k(r))$ intersects the line $a=0$ at the point with the ordinate $\Lambda(r)$ in view of~\eqref{eq5.5}.

The criterion for non-degeneracy of an orbit $\mathcal{O}_{r,\Lambda(r)}$ from the family $\mathcal{C}_1^\Lambda$ and the fact that non-degenerate orbits $\mathcal{O}_{r,\Lambda(r)}$ are elliptic follow from item~2 of Theorem~\ref{th3.7} (which in turn follows from \cite[Prop.~2~(A) and~(B)(c)]{KO:20}).

3) Note that the ``zero'' isoenergy set $\{H=0\}$ consists of two points $(0,N)$, $(0,S)$ of rank 0 (which are non-degenerate and have the center-center type according to item 1 from above) and a one-parameter family of critical orbits $\mathcal{O}_{r,\Lambda(r)}$, $r\in(0,L)$, of rank 1 (which are non-degenerate and have the elliptic type if $r\notin\{r_1^*,\dots,r_N^*\}$ according to the previous item). Therefore, the Reeb graph of the function $K|_{\{H=0\}}$ has the required form.

The connectedness of the intersection of a small circular neighborhood of a point $p$ in a bifurcation complex with a set of the form $\{h>h(p),\ k=k(p)\}$ follows from \cite[Prop.~2~(B)(f)]{KO:20}.
The theorem is proved.
\end{proof}

\begin{figure}[ht!]
    \begin{minipage}[r]{0.49\linewidth}
        \centering
 	\includegraphics[scale=1.58]{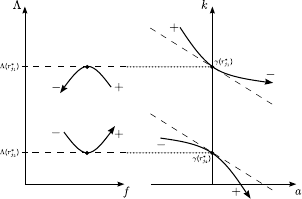} \\ \textit{a)}
    \end{minipage}
    \hfill
    \begin{minipage}[r]{0.49\linewidth}
 	\centering
 	\includegraphics[scale=1.58]{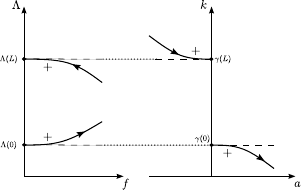} \\ \textit{b)}
    \end{minipage}
    \caption{The behavior of the curves $\Gamma=(f,\Lambda,1)$ and $\gamma=(a,-1,k)$ near the points $r_j^*$ with $\Lambda'(r_j^*)=0$: 
    \textit{a)} $0<r_j^*<L$,
    \textit{b)} $r_0^*=0$, $r_N^*=L$} 
    \label{fig_7}
\end{figure}

\begin{figure}[ht!] 
    \begin{minipage}[r]{0.98\linewidth}
        \centering
        \includegraphics[scale=1.58]{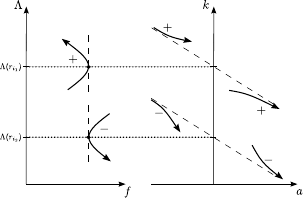} \\
    \end{minipage}
    \caption{The behavior of the curves $\Gamma=(f,\Lambda,1)$ and $\gamma=(a,-1,k)$ near the points $r_i$ with $f'(r_i)=0$}
    \label{fig_8}
\end{figure}

\begin{figure}[ht!] 
    \begin{minipage}[c]{0.48\linewidth}
 	\centering
 	\includegraphics[scale=1.46]
            {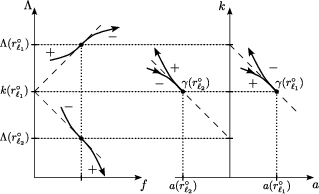} \\ \textit{a)}
    \end{minipage}
    \hfill
    \begin{minipage}[c]{0.48\linewidth}
 	\centering
        \includegraphics[scale=1.46]
        {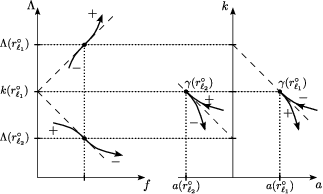} \\ \textit{b)}
    \end{minipage}
    \caption{The behavior of the curves $\Gamma=(f,\Lambda,1)$ and $\gamma=(a,-1,k)$ near the points $r_\ell^\circ$ with $(f'\Lambda''-f''\Lambda')|_{r=r_\ell^\circ}=0$: 
    \textit{a)} two options with $k(r_{\ell_1}^\circ)=k(r_{\ell_2}^\circ)$, \textit{b)} two more options}
    \label{fig_9}
\end{figure}

\begin{remark} \label{rem5.15}
   Clearly, the curve $\gamma=(a,-1,k)$ consists of open arcs $\gamma_i=\gamma|_{(r_{i-1},r_i)}$ separated by the points $P\gamma(r_i):=(1:0:-f(r_i))$ ``at infinity''. These arcs are divided into sub-arcs by cusp points and intersection points with the line $h=0$, i.e., $a=0$.
   Item~2 of Theorem~\ref{th5.13} describes the type of critical orbits corresponding to these sub-arcs. In particular, for $r\in(0,\varepsilon)$ or $r\in(L-\varepsilon,L)$, the critical orbit $\mathcal{O}_{r,k(r)}$ is elliptic (Fig.~\ref {fig_7},~\textit{b}). 
   Two neighboring sub-arcs of the same arc 
   (Fig.~\ref {fig_9} and~\ref {fig_7},~\textit{a}) correspond to different types of critical orbits. Two neighboring sub-arcs of different arcs (separated by a point ``at infinity'') correspond to the same type of critical orbits (Fig.~\ref {fig_8}).
The sign of the Gaussian curvature of the above-mentioned surfaces (see 
footnote$^{\ref{footnote1}}$) is closely related to the sign of the curvature of the planar curves $(f,\Lambda)$ and $(a,k)$, which are meridians of these surfaces of revolution. The behavior of the latter sign is studied in Remark~\ref{rem5.9} and Corollary~\ref{cor5.17}.
\end{remark}

By Lemma~\ref {le5.11}, the curve $\Gamma$ has properties (i)--(iv) characterizing its behavior relatively the points $r_i,r_j^*,r_\ell^\circ$ (i.e., points with a vertical or horizontal tangent line and inflection points). By Lemmas~\ref {le5.12} and~\ref {le5.13}, the curves $P\Gamma=(f:\Lambda:1)$ and $P\gamma=(a:-1:k)$ are projectively dual. Therefore, properties (i)--(iv) of the curve $\Gamma$ can be reformulated in terms of the curve $\gamma$. This is done in the following lemma.

The behavior of the curves $\Gamma$ and $\gamma$ near the points $r_i,r_j^*,r_\ell^\circ$ is shown in Fig.~\ref{fig_7}--\ref{fig_9}, where the arcs are marked with the sign of the Gaussian curvature of the surface of revolution obtained by rotating the arc around the ordinate axis (this sign describes the dynamical properties of the arc, see Theorem~\ref{th5.13}).

\begin{lemma}[Geometric characterization of the curve $\gamma=(a,-1,k)$] \label{le5.16}
Suppose that $P\Gamma=(f:\Lambda:1)$ and $P\gamma=(a:-1:k)$ are projective curves such that $(P\Gamma)^*=P\gamma$ on $r\in I=[0,L]\setminus\{r_i\}$.
Then the fulfillment of conditions {\rm (i)--(iv)} from Lemma~\ref {le5.11}  for the planar curve $\Gamma=(f,\Lambda,1)$ is equivalent to the fulfillment of the following conditions {\rm (i$^*$)--(iv$^*$)} for the planar curve $\gamma=(a,-1,k)$:

\begin{enumerate}
\item[\rm(i$^*$)] The projective curve $P\gamma(r) = (a(r) : -1 : k(r))$, $r\in I=[0,L]\setminus\{r_i\}$, extends by continuity to the punctured points $r_i\in(0,L)$ by points $P\gamma(r_i) = (1: 0: -f_i)$ ``at infinity'' with $f_i>0$ to a good projective curve $P\gamma(r)$, $0\le r\le L$, having no inflection points;

\item[\rm(ii$^*$)] the tangent line to the projective curve $P\gamma$ at any point $P\gamma(r)$, $r\in\left[0,L\right]$ (including the points $P\gamma(r_i)$ ``at infinity'' and the cusp points $P\gamma(r^\circ_\ell)$) is different from the ``line at infinity'' and is not vertical, moreover this line has a negative slope for any point except for the endpoints $p_0=P\gamma(0)$, $p_1=P\gamma(L)$;

\item[\rm(iii$^*$)] all cusp points of the projective curve $P\gamma(r)$, $0\le r\le L$, are finite and do not lie on the axis $a=0$;

\item[\rm(iv$^*$)] the projective curve $P\gamma(r)$, $0\leqslant r\leqslant L$, is smoothly glued at both endpoints $p_0$ and $p_1$ with its reflection
\begin{equation*}
P\gamma_-(r) = (-a(-r), -1, k(-r)),\quad -L\leqslant r\leqslant0,\ r\notin\{-r_i\},\qquad P\gamma_-(-r_i) = (-1: 0: -f_i),
\end{equation*}
with respect to the vertical axis $a=0$ into a smoothly parametrized closed projective curve such that $|\gamma'(0)|=|\gamma''(0)|$ and
$|\gamma'(L)|=|\gamma''(L)|$ (i.e., the length of the velocity vector of the curve $\gamma$ at the endpoints equals the radius of curvature).
\end{enumerate}
\end{lemma}

\begin{proof}
Suppose that $\Gamma(r)$ is a planar curve satisfying conditions (i)--(iv) of Lemma~\ref{le5.11}. Let $r_i,r_\ell^\circ$ be the zeros of $f'(r)$ and $f'(r)\Lambda''(r)-f''(r)\Lambda'(r)$ respectively. Let $(P\Gamma)^* = P\gamma$.

Let us prove (i$^*$).
Since $P\Gamma$ is good and regular by condition (i), then its projectively dual curve $(P\Gamma)^*$ is also good and has no inflection points by Lemma~\ref{le5.8}.
Let us verify that the projective curve $(P\Gamma)^*$ coincides with the union of the curve $P\gamma|_I = (a: -1: k)$ and a finite set of points $P\gamma(r_i) = (1: 0: -f(r_i))$ ``at infinity'', $i=1,\dots,n$. To do this, let us note that the intersection of the projective curve $(P\Gamma)^*$ with the projective line ``at infinity'' formed by points $(x:0:z)$ consists of the points
$(P\Gamma)^*(r) = (-\Lambda'(r) : f'(r) : f(r)\Lambda'(r) - f'(r)\Lambda(r))$ such that $f'(r)=0$, i.e., of the points $(P\Gamma)^*(r_i) = (1 : 0 : -f(r_i))$, $i=1,\dots,n$, due to the non-singularity of $\Gamma$. 
It follows from (iv) that $\Lambda'(0)=\Lambda'(L)=0$, therefore $0<r_i<L$. From this and (ii) we obtain that $f_i=f(r_i)>0$.

Let us prove (iii$^*$). It follows from Lemma~\ref{le5.8} that the cusp points of the curve $(P\Gamma)^*$ correspond to the inflection points of the curve $\Gamma$, i.e., to the parameter values $r_\ell^\circ$.
If the point $P\gamma(r_\ell^\circ)$ were infinite (i.e., had the form $(x:0:z)$) then $f'(r_\ell^\circ)=0$ due to~\eqref{eq5.6}, which means that the tangent line to the curve $\Gamma$ at the point $\Gamma(r_\ell^\circ)$ is vertical, and therefore contradicts the condition (iii). If the point $P\gamma(r_\ell^\circ)$ belonged to the axis $a=0$, then by the construction of the function $a(r)$ in \eqref {eq5.3}, we would have $\Lambda'(r_\ell^\circ)=0$, which means that the tangent line to the curve $\Gamma$ at the point $\Gamma(r_\ell^\circ)$ is horizontal, and therefore contradicts the condition (iii).

Let us prove (ii$^*$). By Lemma~\ref{le5.13}, the tangent line to the projective curve $P\gamma$ at any its point $P\gamma(r)$ is
given by the equation 
$k = -f(r) a + \Lambda(r)$. Hence, this line is not vertical, 
different from the ``line at infinity'' and its slope equals $-f(r)$. It remains to note that $f(r)>0$ for $0<r<L$ in view of (ii).

Let us prove (iv$^*$). We obtain from (iv) that the functions $f,\Lambda$ can be extended to $C^\infty$-smooth odd and even (respectively) $2L$-periodic functions $f_\pm,\Lambda_\pm$ on the real line $\mathbb R$, and therefore $\Lambda'(0)=0$, thus $|f'(0)|=|\Gamma'(0)|=1$. From here and from the formulas \eqref {eq5.5} for $a,k$ in terms of $f,\Lambda$ in Lemma~\ref{le5.12}, we obtain that the functions $a,k$ can be extended to $C^\infty$-smooth odd and even (respectively) $2L$-periodic functions $a_\pm,k_\pm$ on the open subset $I_\pm=\cup_{j\in\mathbb Z}(2Lj+I)\cup(2Lj-I)\subseteq\mathbb R$ containing our point $r=0$.
Hence, the functions $a_\pm,k_\pm$ are smooth in a neighborhood of the point $r=0$. It remains to prove that $|\gamma'(0)|=|\gamma''(0)|$.
From the above formulas for $a,k$ in terms of $f,\Lambda$, taking into account the equalities $\Lambda'(0)=\Lambda'''(0)=0$ and $|f'(0)|=1$, we obtain $a'(0)=\Lambda''(0)/f'(0)=\pm\Lambda''(0)$ and $k''(0)=-\Lambda''(0)=\mp a'(0)$. Hence, in view of the evenness of $k(r)$ and the oddness of $a(r)$, we obtain $|\gamma'(0)|=|a'(0)|=|k''(0)|=|\gamma''(0)|$.
The endpoint $r=L$ is considered similarly.
	
Let us prove the converse. Suppose that $P\gamma(r)$ is an arbitrary parametrized projective curve satisfying conditions (i$^*$)--(iv$^*$). According to (i$^*$), it is good and has no inflection points.
By Lemma~\ref{le5.8}, the curve $(P\gamma)^*$ is good and regular.
It follows from (ii$^*$) that the tangent line to the curve $P\gamma$ at any point $P\gamma(r)$ is given by the equation
$k = - f(r)a + \Lambda(r)$ in the affine chart $(a:-1:k)$,
for some functions $f(r)$ and $\Lambda(r)$, $r\in[0,L]$, where $f(r)>0$ for any $r\in(0,L)$. From the geometric interpretation of a projectively dual curve (see~\eqref{eq5.7}) we have $(P\gamma)^*=(f:\Lambda:1)=P\Gamma$, where $\Gamma(r):=(f(r),\Lambda(r),1)$. As proved, the curve $P\Gamma=(P\gamma)^*$ is good and regular, thus $\Gamma$ satisfies condition (i). 
By Lemma~\ref{le5.8}, its projectively dual curve is $(P\Gamma)^*=(P\gamma)^{**} = P\gamma$.
Since the planar curve $\Gamma=(f,\Lambda,1)$ is smooth and parametrized by a segment, it is bounded and therefore satisfies condition (ii).

Let us derive (iii) from (iii$^*$). Suppose that the tangent line to the curve $\Gamma$ at its inflection point $\Gamma(r_\ell^\circ)$ is vertical or horizontal. Consider the cusp point $P\gamma(r_\ell^\circ)$ of the curve $P\gamma$ that corresponds to it by Lemma~\ref{le5.8}. By \eqref {eq5.5} in Lemma~\ref{le5.12}, we have $a=\infty$ or $a=0$, which contradicts condition (iii$^*$).

Let us derive (iv) from (i$^*$)--(iv$^*$). We obtain from (iv$^*$) that the functions $a,k$ can be extended to $C^\infty$-smooth odd and even (respectively) $2L$-periodic functions $a_\pm,k_\pm$ on $I_\pm = \cup_{j\in\mathbb Z} (2Lj + I)\cup(2Lj - I)$. The point $p_0=P\gamma(0)$ belongs to the axis $a=0$ due to (iv$^*$), and it is finite due to (i$^*$). From this and from the goodness of $P\gamma$ (see (i$^*$)), we obtain that the functions $a_\pm,k_\pm$ are smooth in a neighborhood of the point $r=0$, i.e., $0\in I_\pm$. If $a'(0)=0$ then, from the oddness of $a_\pm$ and the evenness of $k_\pm$, the smoothness of these functions in a neighborhood of zero, and the goodness of the curve $P\gamma$ (see (i$^*$)), it follows that the point $p_0=P\gamma(0)$ is a cusp point of the curve $P\gamma$, which contradicts the condition (iii$^*$). Therefore $a'(0)\ne0$. From this and from the formulas~\eqref{eq5.8} for $f,\Lambda$ in terms of $a,k$ from Lemma~\ref{le5.13}, we obtain that the functions $f,\Lambda$ can be extended to $C^\infty$-smooth odd and even (respectively) $2L$-periodic functions $f_\pm,\Lambda_\pm$ on $\mathbb R$. 

It remains to show that $|\Gamma'(0)|=|\Gamma'(L)|=1$. Since $a'(0)\ne0$, then 
$f'=k'a''/(a')^2 - k''/a'$, due to the above formula for $f$ in terms of $a,k$. From this and from (iv$^*$), we obtain $|\Gamma'(0)| = |f'(0)| = |k''(0)/a'(0)| = |\gamma''(0)|/|\gamma'(0)| = 1$, as required. The point $r=L$ is considered similarly.

Lemma~\ref{le5.16} is completely proved.
\end{proof}

\begin{corollary}[Geometry of the curve $P\gamma(r)$, its intersection with lines $a=0$ and ``at infinity''] \label{cor5.17}
Let $\gamma(r)$ be a curve satisfying conditions {\rm (i$^*$)--(iv$^*$)} from Lemma~\ref{le5.16}, and $\Gamma(r)$ be the corresponding curve. Then:
\begin{itemize}
\item[\rm(a)] the projective curve $P\gamma$ is good and transversally intersects the line $a=0$ at finitely many points $P\gamma(r_j^*) = (0:-1:\Lambda(r_j^*))$ that are regular points of the curve $P\gamma$, $r_1^*=0 < r_2^*<\dots<r_N^*=L$;

\item[\rm(b)] the projective curve $P\gamma$ transversally intersects the ``line at infinity'' (with respect to the affine chart $(a:-1:k)$) at finitely many points $P\gamma(r_i) = (1: 0: -f(r_i))$ that are regular points of the curve $P\gamma$ and do not belong to the line $a=0$, where $0<r_1<\dots<r_n<L$;

\item[\rm(c)] on each arc $\gamma_1=\gamma|_{[0,r_1)}$, $\gamma_i=\gamma|_{(r_{i-1},r_i)}$, $\gamma_{n+1}=\gamma|_{(r_n,L]}$ of the affine curve $\gamma=(a,-1,k)$, the slope is negative  $k'(r)/a'(r)=-f(r)<0$ everywhere except for the points $r\in\{0,L\}$; the oriented curvature
\begin{equation}
\label{eq5.9}
\dfrac{a'k''-a''k'}{(a'^2+k'^2)^{3/2}} =
\dfrac{(k'/a')' a'^2}{(a'^2+k'^2)^{3/2}} =
-\dfrac{f'}{|a'| (1+f^2)^{3/2}}
\end{equation}
does not vanish, has a constant sign, $\sgn(-f')$, is unbounded near
each cusp point $\gamma(r_\ell^\circ)$, and has different signs on adjacent arcs $\gamma_{i-1}$ and $\gamma_i$;

\item[\rm(d)] the tangent line to the projective curve $P\gamma$ at any of its points $P\gamma(r_i)$ ``at infinity'' is a two-sided asymptote of the form $f(r_i) a - \Lambda(r_i) + k = 0$ for two unbounded planar arcs $\gamma|_{(r_i-\varepsilon,r_i)}$ and $\gamma|_{(r_i,r_i+\varepsilon)}$ where $\varepsilon>0$ is sufficiently small; these arcs lie in different half-planes relative to this asymptote.
\end{itemize}
\end{corollary}

\begin{proof}
(a) By Lemma~\ref{le5.12}, $a(r)=\Lambda'(t)/f'(r)$. Therefore, the condition $a(r)=0$ is equivalent to the condition $\Lambda'(r)=0$, i.e., $r=r_j^*$. From~\eqref{eq5.6} and Lemma~\ref{le5.12} we obtain $P\gamma=(P\Gamma)^* = P(\Gamma^*) = (-\Lambda':f':f\Lambda'-f'\Lambda)$. Hence, the point under consideration has the required form $P\gamma(r_j^*) = (0:f'(r_j^*):-f'(r_j^*)\Lambda(r_j^*)) = (0:-1:\Lambda(r_j^*))$, and is therefore finite due to (i$^*$).
Here we used that $f'(r_j^*)\ne0$ due to the regularity of the curve $\Gamma=(f,\Lambda,1)$ (due to condition (i) of Lemma~\ref{le5.11}). The regularity of the point $P\gamma(r_j^*)$ of the curve $P\gamma$ follows from (iii$^*$). The point $P\gamma(r_j^*)$ is a transversal intersection point of the curve $P\gamma$ and the line $a=0$, since this line is vertical and therefore not tangent to the curve $P\gamma$ due to (ii$^*$).

(b) The regularity of the point $P\gamma(r_i)$ of the curve $P\gamma$ follows from (iii$^*$). The point $P\gamma(r_i)$ does not belong to the line $a=0$ due to the finiteness of the intersection points of this line with the curve $P\gamma$ as proved above. The curve $P\gamma$ transversally intersects the ``line at infinity'' at the point $P\gamma(r_i)$, since the tangent line to this curve is different from the ``line at infinity'' due to (ii$^*$).
Both points $P\gamma(r_j^*)$ and $P\gamma(r_i)$ are not inflection points of the curve $P\gamma$, since it has no inflection points due to (i$^*$).

(c) The slope of $\gamma_i$ is negative due to (ii$^*$). 
The latter expression in the formula~\eqref{eq5.9} 
for the oriented curvature 
follows from the formula $f(r)=-k'(r)/a'(r)$ from Lemma~\ref{le5.13}. 
This expression shows that the oriented curvature does not vanish and has a constant sign $\eta=\sgn(-f')$ on each arc $\gamma_i$ of the curve $\gamma$. It also shows that the oriented curvature tends to $\eta\infty$ at points $\gamma(r)$ such that $a'(r)=0$, i.e., at the cusp points $\gamma(r^\circ_\ell)$ of the curve $\gamma=(a,-1,k)$ (in view of the formula $f(r)=-k'(r)/a'(r)$), i.e., at the inflection points of the curve $\Gamma=(f,\Lambda,1)$. It remains to note that the zeros of the function $f'(r)$ are simple, since $f$ is Morse by Lemmas \ref {le5.11} and \ref {le5.16}.
Therefore, $f'(r)$ and the curvature change sign when passing through a point $r_i$.

(d) The equation of the tangent line to $P\gamma$ is obtained in Lemma~\ref{le5.13}. Since the projective curve $P\gamma(r)$ is good, then by Remark~\ref{rem5.10} its tangent line depends smoothly on the parameter. In particular, the tangent line at the point $P\gamma(r_i)$ ``at infinity'' is the limit of the tangent lines at nearby points. Hence, it is a common asymptote for the arcs $\gamma|_{(r_i-\varepsilon,r_i)}$ and $\gamma|_{(r_i,r_i+\varepsilon)}$. Let us show that this asymptote is two-sided and these arcs lie in different half-planes with respect to this asymptote. For this, we have to check that the functions $a(r)$ and $k(r)$, as well as the composition of the parametrized curve $\gamma(r)=(a(r),-1,k(r))$ and the linear function $f(r_i)a-\Lambda(r_i)+k$, which determines the asymptote equation, change sign at the point $r=r_i$.

Since $r_i$ is a simple zero of the function $f'(r)$, then $\zeta:=\Lambda'(r_i)f''(r_i)\ne0$, and from \eqref{eq5.5} we have 
\begin{equation*} 
\begin{split} 
&f(r_i) a(r)- \Lambda(r_i) + k(r) = (f(r_i) - f(r))\dfrac{\Lambda'(r)}{f'(r)} + \Lambda(r)- \Lambda(r_i)
= \frac{\Lambda'(r_i)+o_1}{2} (r-r_i)
\sim \frac{\Lambda'(r_i)}{2} (r-r_i),\\ 
&a( r) = \frac{\Lambda'(r_i)+o_2}{(r-r_i)f''(r_i)} 
\sim \frac{\Lambda'(r_i)}{(r-r_i)f''(r_i)} \to \pm \zeta\infty,
\qquad
k(r) = -\dfrac{f(r_i)\Lambda'(r_i)+o_3}{(r-r_i)f''(r_i)}
\sim -\dfrac{f(r_i)\Lambda'(r_i)}{(r-r_i)f''(r_i)}\to \mp \zeta\infty
\end{split}
\end{equation*}
as $r\to r_i\pm0$, where $o_j=O(r-r_i)$.
Thus, all three functions change sign at the point $r=r_i$, as required.

Corollary \ref {cor5.17} is completely proved.
\end{proof}

One also can show that (iv$^*$) and the properties (a)--(d) from Corollary \ref {cor5.17} are equivalent to (i$^*$)--(iv$^*$).

\subsection{The criterion for realizability of a planar curve as the bifurcation curve of a system $S(f,\Lambda)$}

We obtain from Lemmas~\ref{le5.8}, \ref {le5.11} and \ref{le5.16} that the conditions (i$^*$)--(iv$^*$) in Lemma~\ref{le5.16} on the planar curve $\gamma(r) = (a(r), -1, k(r))$, $r\in I=[0,L] \setminus \{r_1,\dots,r_n\}$, are necessary and sufficient for realizability of this curve as the curve obtained from the bifurcation curve $\gamma_1(r)=(h(r),k(r))$ of a magnetic geodesic flow $S(f,\Lambda)$ using the construction~\eqref{eq5.3}, for some pair of functions $(f,\Lambda)$ satisfying the conditions 1--6 from Assumption~\ref {assump2.4}.
Furthermore, if such a pair of functions $(f,\Lambda)$ exists then it is unique. Moreover, the assignment $\gamma \mapsto \Gamma=(f,\Lambda,1)$ is equivariant with respect to simultaneous admissible changes of parameter $\widetilde r\mapsto r=r(\widetilde r)$ of the curves $\gamma(r)$ and $\Gamma(r)$. Here, by an {\em admissible parameter change} we mean a regular parameter change $[0,\widetilde L]\to[0,L]$, $\widetilde r\mapsto r=r(\widetilde r)$, such that $dr(0)/d\widetilde r=dr(\widetilde L)/d\widetilde r=\pm1$ and the function $dr(\widetilde r)/d\widetilde r$ extends to a $C^\infty$-smooth $2\widetilde L$-periodic even function on $\mathbb R$.

Thus, we obtain the following theorem, which in fact gives a description of all possible bifurcation diagrams for magnetic geodesic flows $S(f,\Lambda)$ on the sphere. We will illustrate it by giving two examples of the curve $\Gamma=(f,\Lambda,1)$ determining the system and the corresponding bifurcation curve $\gamma_1=(h,k)=(a^2/2,k)$ (see Fig.~\ref{fig_10} and Remark~\ref {rem5.18}).

\begin{theorem}  \label{th5.17}
Suppose that a curve $\gamma(r)=(a(r),-1,k(r))$ is given, $r\in I=[0,L] \setminus \{r_1,\dots,r_n\}$, for some points $0<r_1<\dots<r_n<L$. This curve 
satisfies conditions {\rm(i$^*$)}--{\rm(iv$^*$)} from Lemma~\ref{le5.16} if and only if there exists a pair of functions $f(r),\Lambda(r)$, $r\in[0,L]$, satisfying conditions 1--6 from Assumption~\ref {assump2.4} such that the curve $\gamma=(a,-1,k)$ coincides with the curve obtained from the bifurcation curve $\gamma_1=(h,k)$ of the magnetic geodesic flow $S(f,\Lambda)$ using the construction $(h,k)\to(\pm\sqrt{2h},-1,k)=(a,-1,k)$ from~\eqref{eq5.3}. 

Suppose that a planar curve $\gamma$ is realized by a system $S(f,\Lambda)$, for some pair of functions $(f,\Lambda)$. Then:
\begin{itemize}
\item[\rm(a)] such a pair of functions $(f,\Lambda)$ is unique and can be expressed in terms of the curve $\gamma=(a,-1,k)$ by the formulas~\eqref {eq5.8} from  Lemma~\ref{le5.13} everywhere except for a finite number of points $r_i,r_\ell^\circ$; geometrically, these formulas mean that the tangent line to the curve $P\gamma$ at a point $P\gamma(r)$ is given by the equation $f(r)a - \Lambda(r) + k = 0$;
\item[\rm(b)] for any admissible parameter change $r=r(\widetilde r)$, the 
curve $\widetilde\gamma(\widetilde r) := \gamma(r(\widetilde r))$ is realized
by the system $S(\widetilde f,\widetilde\Lambda)$ where $\widetilde f(\widetilde r) := f(r(\widetilde r))$,
$\widetilde\Lambda(\widetilde r) := \Lambda(r(\widetilde r))$.
\end{itemize}
\end{theorem}

\begin{proof}
Let us prove the ``only if'' part of the theorem. Suppose that we are given a curve $\gamma=(a,-1,k)$ satisfying conditions (i$^*$)--(iv$^*$). By Lemmas~\ref {le5.8} and~\ref{le5.16}, there exists a curve $\Gamma=(f,\Lambda,1)$ satisfying conditions (i)--(iv) such that $(P\Gamma|_I)^*=P\gamma$. But, by Lemma~\ref{le5.11}, conditions (i)--(iv) are equivalent to conditions 1--6 from Assumption~\ref {assump2.4}. 
Recall that the conditions 1--6 were imposed on the pair of functions $f,\Lambda$ when defining the magnetic geodesic flow $S(f,\Lambda)$. By Lemma~\ref{le5.12}, the curve $P\gamma = (P\Gamma|_I)^*$ can be restored from the bifurcation curve $\gamma_1$ of the system $S(f,\Lambda)$ using the construction~\eqref {eq5.3}.

Let us prove the ``if'' part of the theorem. Suppose that a curve $\Gamma=(f,\Lambda,1)$ satisfies conditions 1--6 from Assumption~\ref{assump2.4}.
By virtue of Lemma~\ref{le5.11}, these conditions are equivalent to conditions (i)--(iv) from Lemma~\ref{le5.11}. Suppose that the curve $\gamma=(a,-1,k)$
is obtained from the bifurcation curve $\gamma_1$ of the system $S(f,\Lambda)$ using the construction~\eqref {eq5.3}. By Lemma~\ref{le5.12}, the projective curve $P\gamma$ is projectively dual to the curve $P\Gamma$, and by
Lemma~\ref{le5.16} it satisfies conditions (i$^*$)--(iv$^*$).

By Lemma~\ref{le5.13}, the formulas~\eqref {eq5.8} are true. These formulas express the curve $\Gamma=(f,\Lambda,1)$ in terms of the curve $\gamma$ everywhere, except for a finite number of points $r_i,r_\ell^\circ$. Therefore, the pair $(f,\Lambda)$ is unique.

It remains to note that under change of parameter of the curve $\Gamma$,  its projective dual curve $P\gamma = (P\Gamma)^*$ is transformed by the same change of parameter (this follows from the definitions \ref{def5.3}--\ref{def5.5} and Lemma~\ref{le5.8}).

The theorem is proved.
\end{proof}

\begin{figure}[ht!] 
    \textit{a)}
    \begin{minipage}[r]{0.98\linewidth}
		\centering
	\includegraphics[scale=1.85]{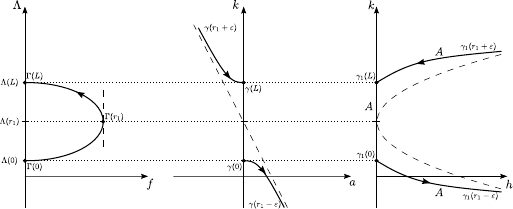} 
    \end{minipage}
    \textit{b)}
    \begin{minipage}[r]{0.98\linewidth}
        \centering
		\includegraphics[scale=1.85]{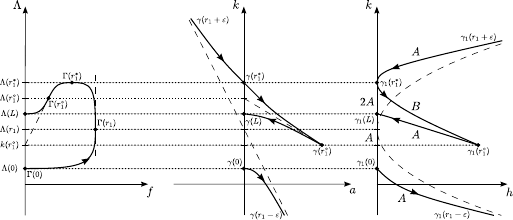}
    \end{minipage}
    \caption{The curves $\Gamma=(f,\Lambda,1)$, $\gamma=(a,-1,k)$ and $\gamma_1=(h,k)=(a^2/2,k)$ for the system $S(f,\Lambda)$ on a sphere, examples: 
    \textit{a)} a uniform magnetic field on a round sphere, 
    \textit{b)} a non-uniform magnetic field}
    \label{fig_10}
\end{figure}

\begin{remark}[about Fig.~\ref{fig_10}] \label {rem5.18}
Two examples of the corresponding curves $\Gamma=(f,\Lambda,1)$, $\gamma=(a,-1,k)$ and $\gamma_1=(h,k)=(a^2/2,k)$ for the magnetic geodesic flows $S(f,\Lambda)$ on a sphere are shown in Fig.~\ref{fig_10}. The bifurcation diagram is contained in the half-plane $h\ge0$ and consists of two points $(0,\Lambda(0))$, $(0,\Lambda(L))$ and the following curves:
\begin{itemize}
\item the curve $\gamma_1$ (oriented by the parameter $r\in I\subset[0,L]$), which has ``two-sided asymptotes--parabolas'' $\{h=U_{\Lambda(r_i)}(r_i)\}$ shown as dashed lines,
\item the interval $\gamma_2\subset\{h=0\}$, see Theorem \ref {th3.3}.
\end{itemize}
All the asymptotes--parabolas are tangent to the axis $h=0$ at their vertices.
The labels $A,B$ on the arcs of the bifurcation diagram in the plane $(h,k)$ describe the topology of the Liouville foliation near the corresponding 1-parameter family of non-degenerate singular fibers (elliptic 3-atom $A$, hyperbolic 3-atom $B$).
\end{remark}

\section*{Acknowledgement}
The authors are grateful to Anatoly Fomenko for setting the problem, to Alexey Bolsinov and Andrey Oshemkov for helpful discussions, and to the anonymous referee for valuable comments that helped to improve the presentation.

\section*{Funding}
The work by E.~A.~K.\ was performed under the development programme of the Volga Region Mathematical Center (agreement no.~075-02-2024-1438) and supported by the Russian Science Foundation, grant No.~24-71-10100.






\begin{thebibliography}{99}

\bibitem{IGS}
\by A. V. Bolsinov and A. T. Fomenko
\book Integrable Hamiltonian Systems: Geometry, Topology, Classification
\bookinfo Vol.~I
\publaddr Boca Raton, London, N.Y., Washington
\publ D.C.\ Chapman~\&~Hall\,/\,CRC
\bookinfo (Engl. transl. of Russian version: Udmurtskiy universitet, Izhevsk, 1999)
\yr 2004 

\bibitem{Fom:86_1}
\by A. T. Fomenko 
\paper Morse theory of integrable Hamiltonian systems
\jour Soviet Math.  Dokl.
\vol 33 
\issue 2 
\pages 502--506 
\yr 1986

\bibitem{Fom:86_2}
\by A. T. Fomenko 
\paper Topology of isoenergy surfaces of integrable Hamiltonian systems and obstructions to integrability
\jour Math. USSR Izv.
\vol 29 
\issue 3 
\pages 629--658   
\yr 1987 

\bibitem{Fom:87}
\by A. T. Fomenko and H. Zieschang
\paper On topology of three-dimensional manifolds arising in Hamiltonian mechanics
\jour Soviet Math.  Dokl.
\vol 35
\issue 2
\pages 520--534
\yr 1987

\bibitem{Fom:90}
\by A. T. Fomenko and H. Zieschang
\paper Topological invariant and a criterion of equivalence of  integrable Hamiltonian systems with two degrees of freedom
\jour Math. USSR Izv.
\vol 36
\issue 3
\pages 567--596
\yr 1990

\bibitem{BRF:2000}
\by A.~V.~Bolsinov, P.~H.~Richter and A.~T.~Fomenko
\paper The method of loop
molecules and the topology of the Kovalevskaya top
\jour Sb.\ Math.
\vol 191
\issue 2
\pages 151--188
\yr 2000

\bibitem{Kant}
\by E. O. Kantonistova 
\paper Topological classification of integrable Hamiltonian systems in a~potential field on surfaces of revolution
\jour Sb. Math.
\vol 207
\issue 3
\pages 358--399 
\yr 2016

\bibitem{Tim}
\by D. S. Timonina 
\paper Liouville classification of integrable geodesic flows in a~potential field on two-dimensional manifolds of revolution: torus and Klein bottle
\jour Sb. Math.
\vol 209
\issue 11
\pages 1644--1676 
\yr 2018

\bibitem{KZAnt}
\by E. I. Antonov and I. K. Kozlov 
\paper Liouville classification of integrable geodesic flows in a potential field on a projecive plane 
\jour Che\-by\-shev\-skii Sb.
\vol 21
\issue 2
\pages 10--25
\yr 2020

\bibitem{KO:20}
\by E.~A.~Kudryavtseva and A.~A.~Oshemkov
\paper Bifurcations of
integrable mechanical systems with magnetic field on surfaces of revolution
\jour Che\-by\-shev\-skii Sb.
\yr 2020
\vol 21
\issue 2
\pages 244--265
\lang Russian

\bibitem{KK:25}
\by I. F. Kobtsev and E. A. Kudryavtseva 
\paper Bifurcations of magnetic geodesic flows on toric surfaces of revolution 
\jour Che\-by\-shev\-skii Sb.
\yr 2025
\paperinfo (in print)

\bibitem{VP:2023}
\by V. V. Vedyushkina and S. E. Pustovoytov
\paper Classification of Liouville foliations of integrable topological billiards in a magnetic field
\jour Sb. Math.
\vol 214
\issue 2
\pages 23--57
\yr 2023

\bibitem{KZO:E3}
\by I. Kozlov and A. Oshemkov 
\paper Integrable systems with linear periodic integral for the Lie algebra $e(3)$
\jour Lobachevskii J. Math.
\vol 38
\pages 1014--1026
\yr 2017

\bibitem{VKZ:79}
\by V. V. Kozlov
\paper Topological obstructions to integrability of natural mechanical systems
\jour Soviet Math.  Dokl.
\yr 1979
\vol 249
\issue 6
\pages 1299--1302

 \bibitem{Besse:78}
 \by A. Besse
  \book  Manifolds all of whose geodesics are closed
  \publaddr Berlin-New York 
  \publ Springer-Verlag 
  \yr 1978 

\bibitem{Fom:88}
\by A. T. Fomenko
\paper Topological invariants of Liouville integrable Hamiltonian systems 
\jour Funct.\ Anal.\ Appl.
\vol 22
\issue 4
\pages 286--296
\yr 1988

\bibitem{EG:2012}
\by K. Efstathiou and A. Giacobbe
\paper The topology associated with cusp singular points
\jour Nonlinearity
\vol 25
\pages 3409--3422
\yr 2012

\bibitem{vu2003}
\by S. V\~u Ng\d oc
\paper On semi-global invariants for focus-focus singularities
\jour Topology
\yr 2003
\vol 42
\pages 365--380

\bibitem{BGK:2018}
\by A.~V.~Bolsinov, L.~Guglielmi and  E.~A.~Kudryavtseva
 \paper Symplectic
invariants for parabolic orbits and cusp singularities of integrable systems
\jour Philos. Trans. Roy. Soc. A
\vol 376
\issue 2131
\yr 2018
 \pages 20170424, 29 pp

\bibitem{bol:osh06}
\by A.~V.\ Bolsinov and A.~A.\ Oshemkov
\book Singularities of Integrable Hamiltonian Systems
\publ in {\em Topological Methods in the Theory of Integrable Systems},
eds.\ A.~V.\ Bolsinov, A.~T.\ Fomenko, and A.~A.\ Oshemkov
(Cambridge Scientific Publications, Cambridge, 2006), pp.~1--67
\pages 1--67

\bibitem {Hansmann}
\by H.\ Han\ss mann
\book Local and Semi-Local Bifurcations in Hamiltonian Dynamical Systems -- Results
and Examples
\publ Springer, Lecture Notes in Mathematics, Vol.~1893
\publaddr Berlin, Heidelberg
\yr 2007

\bibitem{zung96a} 
\by N. T. Zung
\paper Symplectic topology of integrable Hamiltonian systems, I: Arnold-Liouville with singularities
\jour Compositio Math.
\yr 1996
\vol 101
\pages 179--215

\bibitem {zung03a}
\by N. T. Zung
\paper Symplectic topology of integrable Hamiltonian systems, II: Topological classification
\jour Compositio Mathematica
\yr 2003
\vol 138
\pages 125--156

\bibitem{Ler:87}
\by L.~M.~Lerman and Ya.~L.~Umanskii
\paper The structure of a Poisson action of
$\mathbb{R}^2$ on a four-dimensional symplectic manifold
\jour Selecta Math.\ Sov.
\paperinfo \ (Engl.\ transl.\ of Russian preprint of 1981)
\vol 6
\issue 4
\pages 365--396
\yr 1987

\bibitem{Ler:94}
\by L.~M.~Lerman and Ya.~L.~Umanskii
\paper Classification of four-dimensional integrable Hamiltonian systems and Poisson actions of $R^2$ in extended neighborhoods of simple singular points. I
\jour Russian Acad. Sci. Sb. Math.
\vol 77
\issue 2
\pages 511--542
\yr 1994

\bibitem{KM:21}
\by E.~A.~Kudryavtseva and N.~N.~Martynchuk
\paper Existence of a smooth
Hamiltonian circle action near parabolic orbits and cuspidal tori
\jour Regular
and Chaotic Dynamics
\vol 26
\issue 6
\pages 732--741
\yr 2021

\bibitem{Ku:22}
\by E.~A.~Kudryavtseva
\paper Hidden toric symmetry and structural stability
of singularities in integrable systems
\jour European Journal of Mathematics
\yr 2022
\vol 8
\pages 1487--1549
  
\bibitem{AVG:85}
\by V.~I. Arnol’d, S.~M. Guse\u\i n-Zade and A.~N. Varchenko
\book Singularities of differentiable maps
\bookinfo Vol.~I
\publaddr Boston, MA
\publ Birkh\"auser
\yr 1985

\bibitem{Tev:2000}
\by E. A. Tevelev
\book Projective Duality and Homogeneous Spaces
\publaddr Berlin, Heidelberg 
\publ Springer 
\yr 2000

\end{thebibliography}
\end{document}